\documentclass[12pt]{amsart}

\usepackage[dvipsnames]{xcolor}
\usepackage[margin=1in]{geometry}
\usepackage{amsmath}
\usepackage{amsfonts}
\usepackage{amssymb}
\usepackage[hidelinks]{hyperref}
\usepackage{graphicx}
\usepackage[inline]{enumitem}
\usepackage{ulem}

\usepackage[backend=biber,style=alphabetic]{biblatex}
\usepackage{amsthm}
\theoremstyle{plain}
\newtheorem{thm}{Theorem}[section]
\newtheorem{lem}[thm]{Lemma}
\newtheorem{prop}[thm]{Proposition}
\newtheorem{cor}[thm]{Corollary}

\theoremstyle{definition}
\newtheorem{defn}[thm]{Definition}
\newtheorem{conj}[thm]{Conjecture}
\newtheorem{exmp}[thm]{Example}
\newtheorem{question}[thm]{Question}

\usepackage{tikz}
\tikzset{
    active_graph_node/.style = {circle, draw=black, fill=black, inner sep = 1.5pt, text=white, minimum size=20pt},
    dormant_graph_node/.style = {circle, draw=black, fill=white, inner sep = 1.5pt, text=black, minimum size=20pt},
    grayfill/.style={gray!20,draw=black}
}

\newcommand{\abs}[1]{\left|#1\right|}
\newcommand{\Conf}{\mathrm{Conf}}
\newcommand{\DConf}{\mathrm{Conf}^{\,\square}}
\newcommand{\UConf}{\mathrm{UConf}}
\newcommand{\DUConf}{\mathrm{UConf}^{\,\square}}

\title{Locally Euclidean Discretized Graph Configuration Spaces}

\date{}

\author{Justin Murray}
\address{Justin Murray, Department of Mathematical Sciences, Ball State University}
\email{justin.murray@bsu.edu}

\author{Zach Wilcher}
\address{Zach Wilcher, Department of Mathematics, University of Florida}
\email{z.wilcher@ufl.edu}

\begin{document}
    \begin{abstract}
        We investigate discretized graph configuration spaces as defined by Aaron Abrams in his PhD thesis in 2000. Specifically, we extend Abrams' work by completely classifying which ones admitted by simple graphs are homeomorphic to closed manifolds. In the process of doing so, we develop techniques to translate topological properties of these spaces into graph theoretic ones.
    \end{abstract}
    \maketitle
    \tableofcontents
    \section{Introduction}\label{section:introduction}
Graph configuration spaces are excellent models for numerous real-world applications,
including safe robot pathing \cite{ghristSafeCooperativeRobot2002},
digital microfluidic devices \cite{GHRIST2007302}, and topological quantum computers \cite{Maciazek2019}; their study has been a growing area of research in geometric settings since \cite{abrams2000configurationspaces}.

A key feature of Abrams' model is that \(\DConf_n(\Gamma)\) is homotopy equivalent to \(\Conf_n(\Gamma)\) when \(\Gamma\) is suitably subdivided (see Theorem 3.2 of \cite{PRUE2014136}). So, by studying \(\DConf_n(\Gamma)\) we can sometimes determine homotopy invariants of \(\Conf_n(\Gamma)\)
such as the graph braid group \(B_n(\Gamma)\). One property of the graph braid group is the possibility of being a manifold group. In \cite{jainManifoldModelsHyperbolic2026} it is shown that \(B_3(\Theta_5)\) is a \(3\)-manifold group. By classifying exactly which discretized graph configuration spaces are manifolds, we partially determine which graph braid groups are manifold groups; any other examples of manifold groups must be determined through other means, e.g.\ thickening.

The goal of this paper is to resolve the following question of Abrams.

\begin{question}[Question 5.1 in \cite{abrams2000configurationspaces}]
    \label{qst:graph-manifolds}
    For which simple graphs \(\Gamma\) without loops and which positive integers \(n\) and \(m\) is the discretized \(n\)-point (ordered) configuration space, denoted \(\DConf_n(\Gamma)\), homeomorphic to a closed \(m\)-manifold without boundary?
\end{question}

Abrams had made the most headway towards answering the general question in \cite{abrams2000configurationspaces}. He classifies all such examples when the dimension of the manifold is maximal. That is, when the number of points in the configuration space equals the dimension.
\begin{thm}[Corollary 5.8 in \cite{abrams2000configurationspaces}]
    If \(\Gamma\) is a simple graph without loops, then \(\DConf_n(\Gamma)\) is homeomorphic to an \(n\)-manifold without boundary
    if and only if \(n = 1\)
    and \(\Gamma\) is a cycle
    or
    \(n = 2\) and \(\Gamma\) is isomorphic to \(K_5\) or \(K_{3,3}\).
\end{thm}

By establishing a duality result, Abrams recovers other examples of manifolds in \cite{abrams2000configurationspaces} such as 
\(\DConf_2(K_{1,3})\), \(\DConf_3(K_5)\), and \(\DConf_4(K_{3,3})\). Our interest in this question started when \(\DConf_3(\Theta_4)\) was shown to be closed orientable surface in \cite{appiah2024algebraicstructurehyperbolicgraph}. This had been done before (see Example 4.3 of \cite{ko2012characteristics}). However, while searching for other examples, we found that such examples are rare.

Interested in generalizing the notion of when \(\DConf_n(\Gamma)\) is homeomorphic to a manifold, Fernandes investigated when \(\DConf_2(\Gamma)\) is a \(2\)-\textit{pseudomanifold} in \cite{fernandesTopologyGraphConfiguration}.

\begin{defn}[Definition 5.1 in \cite{fernandesTopologyGraphConfiguration}]
    \label{defn:fernandes-pseudomanifold}
    A finite \(m\)-dimensional cubical complex \(X\) is a closed \(m\)-pseudomanifold if and only if
    every \((m-1)\)-cell in \(X\) is a face of exactly two \(m\)-dimensional cells.
\end{defn}

Fernandes determined that \(\DConf_2(\Gamma)\) is a closed \(2\)-pseudomanifold if and only if \(\Gamma\) is \(K_5\), \(K_{3,3}\), \(K_4\) with each edge doubled, or \(K_{2,2}\) with each edge doubled. We approach question \ref{qst:graph-manifolds} in a way similar to how Fernandes approached \(\DConf_2(\Gamma)\) by determining graph invariants necessary for \(\DConf_n(\Gamma)\) to be a \(m\)-pseudomanifold. Unlike Fernandes, the graph invariants we determine involve the language of graph \textit{matchings} (see Definition \ref{defn:matching-set}), which generalize naturally for increasing values of \(m\) and \(n\). With this approach, we completely classify when \(\DConf_n(\Gamma)\) is an \(m\)-pseudomanifold and \(\Gamma\) is a simple graph in Theorem \ref{thm:m-pseudomanifolds} below. Note that this theorem contains Fernandes' result (for simple graphs) as a special case. 

\begin{thm}
    \label{thm:m-pseudomanifolds}
    
    %  \(\DConf_n(\Gamma)\) is a closed \(m\)-pseudomanifold without boundary if and only if \(\DConf_n(\Gamma)\) is one of the following. 
    %  
    %\begin{enumerate}
    %   \item \(\DConf_n(\bigsqcup_{k \in I} C_k)\) where \(n \in \{1, \abs{V(\Gamma)} - 1\}\) and \(I\) is any finite multiset of integers each \(\ge 3\) \zw{this is a 1-pseudomanifold. As written it sounds like we are saying this is an \(m\)-pseudomanfold for \(m > 1\)}
    %   \item \(\DConf_n(K_{2m+1})\) where \(n \in \{m, m+1\}\).
    %   \item \(\DConf_n(K_{m+1, m+1})\) where \(n \in \{m, m+2\}\)
    %   \item \(\DConf_{m+1}(K_p \sqcup K_q)\) for some positive odd integers \(p\) and \(q\) satisfying \(p + q = 2m + 2\).
    %   \item \(\DConf_{m+1}(K_{m, m+2})\)
    %\end{enumerate}
    
    \(\DConf_n(\Gamma)\) is a closed \(m\)-pseudomanifold if and only if at least one of the following is true.

    \begin{enumerate}
        \item \(\Gamma\) is a finite disjoint union of cycles, \(n \in \{1, \abs{V(\Gamma)} - 1\}\), and \(m = 1\)
        \item \(\Gamma \cong K_{2m+1}\) and \(n \in \{m, m+1\}\)
        \item \(\Gamma \cong K_{m+1, m+1}\) and \(n \in \{m, m+2\}\)
        \item \(\Gamma \cong K_{m, m+2}\) and \(n = m+1\)
        \item \(\Gamma \cong K_p \sqcup K_q\) for some odd \(p\) and \(q\) satisfying \(p + q = 2m + 2\) and \(n = m+1\)
    \end{enumerate}

\end{thm}
\begin{proof}
See Subsection \ref{subsection:pseudomanifolds}.
\end{proof}

Since all manifolds are pseudomanifolds,
a complete classification of manifolds follows as a corollary. 

\begin{cor}
  \label{cor:manifold-classification}
%\(\DConf_n(\Gamma)\) is homeomorphic to a closed manifold if and only if it is one of the following.
%\begin{enumerate}
%  \item \(\DConf_n(\bigsqcup_{k\in I} C_k)\) where \(n \in \{1,\abs{V(\Gamma} - 1\}\) and \(I\) is any finite multiset of integers each \(\ge 3\)
%  \item \(\DConf_2(K_3 \sqcup K_1)\)
%  \item \(\DConf_2(K_{1,3})\)
%  \item \(\DConf_n(K_{5}) \) where \(n \in \{2,3\}\)
%  \item \(\DConf_n(K_{3,3})\) where \(n \in \{2, 4\}\)
%  \item \(\DConf_3(K_{2,4})\)
%  \item \(\DConf_3(K_5 \sqcup K_1)\)
%\end{enumerate} 

%Moreover, \(\DConf_n(\bigsqcup_{k\in I} C_k)\), \(\DConf_2(K_3\sqcup K_1)\), and \(\DConf_2(K_{1,3})\) are 1-manifolds, whereas the other cases are 2-manifolds.

\(\DConf_n(\Gamma)\) is homeomorphic to a closed \(m\)-manifold if and only if at least one of the following is true.
\begin{enumerate}
    \item \(\Gamma\) is a finite disjoint union of cycles, \(n \in \{1, \abs{V(\Gamma)} - 1\}\), and \(m = 1\)
    \item \(\Gamma \cong K_3\sqcup K_1\), \(n = 2\), and \(m = 1\)
    \item \(\Gamma \cong K_{1,3}\), \(n = 2\), and \(m = 1\)
    \item \(\Gamma \cong K_5\), \(n \in \{2,3\}\), and \(m = 2\)
    \item \(\Gamma \cong K_{3,3}\), \(n \in \{2,4\}\), and \(m = 2\)
    \item \(\Gamma \cong K_{2,4}\), \(n = 3\), and \(m = 2\)
    \item \(\Gamma \cong K_5\sqcup K_1\), \(n = 3\), and \(m = 2\)
\end{enumerate}

\end{cor}

\begin{proof}
See Subsection \ref{subsection:manifolds}.
\end{proof}

The remainder of the article is organized as follows. In Section \ref{section:definitions}, we review the definitions for configuration spaces and establish notation. Section \ref{section:pseudomanifolds} is devoted to establishing several important lemmas that restrict the types of graphs that are permitted. In Section \ref{section:special-subgraphs}, we use the techniques developed in previous sections to reconstruct what the graph has to be. Finally, Section \ref{section:questions} is dedicated to open questions.
    \section{Definitions and Preliminaries}\label{section:definitions}

\subsection{Configuration Spaces}
\begin{defn}
    \label{defn:topological_configuration_space}
The \(n\)-point configuration space of a space \(X\) is
\(\Conf_n(X) = X^n \setminus \Delta\) where \(\Delta\) is the ``fat'' diagonal
\(\{(x_1, \cdots, x_n) \in X^n \mid x_i = x_j \text{ for some } i \neq j\}\).
\end{defn}

When our space is the geometric realization of a graph denoted \(\abs{\Gamma}\), we will often just write \(\Conf_n(\Gamma)\) instead of \(\Conf_n(\abs{\Gamma})\), but both of these symbols refer to the same space. Intuitively, \(\Conf_n(X)\)
is the set of positions of \(n\) particles placed in \(X\) such that
no two particles occupy the same spot.
Given some fixed point \(\vec{x} \in X^n\),
we can define the \(n\)-strand pure braid group of \(X\) as \(P_n(X) = \pi_1(\Conf_n(X), \vec{x})\).
When \(X = \mathbb{R}^2\), this is simply known as the \(n\)-strand pure braid group.
By taking \(\pi_1\) of the \(n\)-point \textit{unordered} configuration space of \(X\),
we obtain the braid group of \(X\) denoted \(B_n(X)\).
The braid group of \(\mathbb{R}^2\) is the classical Artin braid group.
We point the reader to \cite{birmanBraidsSurvey2005} for a comprehensive survey of braids. 

\begin{defn}
\label{defn:topological_unordered_configuration_space}
Given a space \(X\), 
the \(n\)-point unordered configuration space of \(X\)
is \(\UConf_n(X) = \Conf_n(X) / S_n\)
where \(S_n\) is the degree \(n\) permutation group acting on \(\Conf_n(X)\) by permuting the coordinates
of points in \(\Conf_n(X)\).
\end{defn}

Intuitively one can think of \(\UConf_n(X)\) as the set of positions of \(n\) indistinguishable particles placed in \(X\)
such that no two particles occupy the same spot.
More precisely, a permutation \(\sigma\) of the set \(\{1,\cdots, n\}\)
acts on a point \(\vec{x} = (x_1, \cdots, x_n)\) in \(\Conf_n(X)\)
by giving the point \(\sigma \vec{x} = (x_{\sigma(1)}, \cdots, x_{\sigma(n)})\).
The orbit of \(\vec{x}\) under this action is the set \(S_n \vec{x} = \{\sigma \vec{x} \mid \sigma \in S_n\}\),
and the orbit space \(\Conf_n(X) / S_n\) is the set of orbits of points in \(\Conf_n(X)\).

Critically, this action is a \textit{covering space action} (see Proposition 1.40 of \cite{hatcherAlgebraicTopology2001})
meaning \(p: \Conf_n(X) \rightarrow \Conf_n(X) / S_n\) where \(p(\vec{x}) = S_n \vec{x}\)
is a covering map.
Hence by Proposition 1.40 and Proposition 1.32 of \cite{hatcherAlgebraicTopology2001} 
if \(\Conf_n(X)\) and \(\UConf_n(X)\) are both path connected
and locally path connected, then \(\Conf_n(X)\) is an \(n!\)-fold cover of \(\UConf_n(X)\). In this work we primarily work with the ordered configuration space although the unordered space will come in handy to simplify some counting arguments in subsection \ref{subsection:euler}.

\subsection{Graph Configuration Spaces}
In Figure \ref{fig:K2_and_Conf2}, we illustrate \(\Conf_2(\abs{\Gamma})\)
when \(\Gamma \cong K_2\).
The placement of two particles on \(\abs{\Gamma}\) such that the first
is at \(v\) and the second is as \(w\) corresponds
to the point \((v,w)\) in the upper left.
As these two particles move on \(\abs{\Gamma}\), a triangular
region is spanned in the configuration space.
If we instead placed our particles down 
such that the first is at \(w\) and the second is at \(v\),
and let them move, a distinct triangular region is spanned in the lower right.
These two regions are disconnected by the diagonal where
the two particles have the same position.

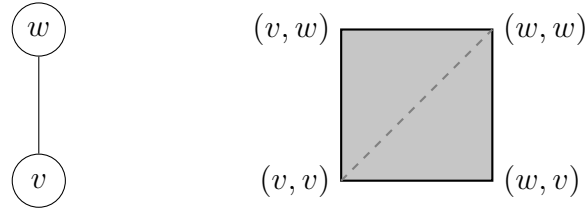
\begin{figure}[h!]
		\centering
        \begin{tikzpicture}[scale=2]
            \begin{scope}
                \node[dormant_graph_node] (v) at (0,0) {\(v\)};
                \node[dormant_graph_node] (w) at (0,1) {\(w\)};
                \draw (v) -- (w);
            \end{scope}
            \begin{scope}[shift={(2,0)}]
                \fill[gray, opacity=0.45] (0,0) rectangle (1,1);
                \draw[thick] (0,0) rectangle (1,1);
                \node[left] (vv) at (0,0) {\((v, v)\)};
                \draw[gray, dashed, thick] (0,0) -- (1,1);
                \node[right] (ww) at (1,1) {\((w, w)\)};
                \node[left] (ww) at (0,1) {\((v, w)\)};
                \node[right] (ww) at (1,0) {\((w, v)\)};
            \end{scope}
        \end{tikzpicture}
        \caption{\(\Gamma \cong K_2\) and \(\Conf_2(\Gamma)\)}
        \label{fig:K2_and_Conf2}
\end{figure}

\begin{defn}[Definition 2.1 in \cite{abrams2000configurationspaces}]
\label{defn:combinatorial_configuration_space}
Given a graph \(\Gamma\), the \(n\)-point discretized configuration space of \(\Gamma\)
is \(\DConf_n(\Gamma) = \abs{\Gamma}^n - \Delta^{\square}\)
where 
\(\Delta^{\square}\)
is the union of every \(n\)-tuple of cells \((e_1, \cdots, e_n)\) in \(\abs{\Gamma}^n\)
where \(\partial e_i \cap \partial e_j \neq \emptyset\) for some \(i \neq j\).
\end{defn}

In the language of placing particles on \(\abs{\Gamma}\), instead of just requiring that no two particles
occupy the same position like in \(\Conf_n(\Gamma)\), we now also require that no two particles occupy ``adjacent'' positions
where two positions in \(\abs{\Gamma}\) are adjacent if the cells \(e_i\) and \(e_j\) that contain them
satisfy \(\partial e_i \cap \partial e_j \neq \emptyset\).

\begin{figure}[h!]
	\centering
    \begin{tikzpicture}[scale=2]
            \begin{scope}
                \node[dormant_graph_node] (v) at (0,0) {\(v\)};
                \node[dormant_graph_node] (w) at (0,1) {\(w\)};
                \draw (v) -- (w);
            \end{scope}
            \begin{scope}[shift={(2,0)}]
                \node[circle, fill, inner sep=1.5pt, label=right:{\((w,v)\)}] (v) at (1,0) {};
                \node[circle, fill, inner sep=1.5pt, label=left:{\((v,w)\)}] (w) at (0,1) {};
            \end{scope}
    \end{tikzpicture}
    \caption{\(K_2\) and \(\DConf_2(K_2)\)}
    \label{fig:K2_and_DConf2}
\end{figure}

In Figure \ref{fig:K2_and_DConf2} we revisit placing two particles on \(\Gamma \cong K_2\).
We illustrate like we did in \(\Conf_2(\Gamma)\) 
what happens when we place the first particle at \(v\) and the second at \(w\), 
by plotting the corresponding point in the top left.
However, if we tried to move either particle off of its vertex, we would find
that the boundary of the edge \(vw\) containing it intersects the vertex containing the other particle.
Hence, both particles are stuck and we have an isolated point in \(\DConf_2(\Gamma)\).
Similarly, placing the first particle at \(w\) and the second at \(v\),
we again find both particles are stuck resulting in another isolated point in \(\DConf_2(\Gamma)\).

In the context of Definitions \ref{defn:topological_configuration_space} and \ref{defn:combinatorial_configuration_space} 
since \(\Delta^{\square}\) is so much larger than \(\Delta\),
one might wonder what the resemblance between \(\Conf_n(\Gamma)\) and \(\DConf_n(\Gamma)\) is.
Abrams showed in Theorem 2.1 of \cite{abrams2000configurationspaces} that \(\Conf_n(\Gamma)\) deformation retracts onto \(\DConf_n(\Gamma)\) 
when \(\Gamma\) is ``sufficiently subdivided'' i.e.\ extra degree-\(2\) vertices are inserted on edges
so that conditions 1. and 2. of the following theorem are satisfied.

\begin{thm}[Theorem 3.2 in \cite{PRUE2014136}]
    \label{thm:suitable_subdivision}
    Let \(n > 1\) and \(\Gamma\) be a connected graph.
    We say a vertex is essential if its degree is not \(2\). 
    \(\Conf_n(\Gamma)\) deformation retracts onto \(\DConf_n(\Gamma)\) if both of the following are true.
    \begin{enumerate}
        \item Each path connecting distinct essential vertices of \(\Gamma\) has length at least \(n-1\).
        \item Each homotopically essential path connecting a vertex to itself has length at least \(n+1\).
    \end{enumerate}
\end{thm}

Abrams originally conjectured that the conditions of this theorem were sufficient, but this was 
first proven in \cite{PRUE2014136}.
As mentioned in the introduction, Theorem \ref{thm:suitable_subdivision} means that by studying \(\DConf_n(\Gamma)\) we can still
determine homotopy invariants of \(\Conf_n(\Gamma)\) such as the fundamental group or homology groups
as long as \(\Gamma\) is sufficiently subdivided.
It is worth mentioning at this point that Abram's discretized graph configuration space 
is not the only ``discretized graph configuration space''. In \cite{Maciazek2019} the authors consider Swiatkowski's combinatorial
graph configuration space. They remark that Swiatkowski's model has the advantage that 
it has the same homological dimension as \(\Conf_n(\Gamma)\) without any necessary subdivisions.
However, they note that Abrams' model is more convenient to use when applying discrete Morse theory.
In this work, we only consider Abram's model and do not subdivide the graph. We note that each space in Corollary \ref{cor:manifold-classification} except \(\DConf_3(K_5)\), \(\DConf_4(K_{3,3})\), and \(\DConf_3(K_5\sqcup K_1)\) is a deformation retract of its ordinary configuration space since if \(n\) is 3 or 4, then the path between the essential vertices needs to be at least 2 or 3 edges, respectively. However, in all of these graphs there is only 1 edge between the essential vertices.

We can also construct the unordered discretized graph configuration space \(\DUConf_n(\Gamma)\)
in essentially the same way as we did in Definition \ref{defn:topological_unordered_configuration_space}.
The same reasoning that follows after Definition \ref{defn:topological_unordered_configuration_space} shows that \(\DConf_n(\Gamma)\) is an \(n!\)-fold cover of \(\DUConf_n(\Gamma)\) when both \(\DConf_n(\Gamma)\) and \(\DUConf_n(\Gamma)\) are path connected. Theorem 2.6 in \cite{abrams2000configurationspaces} shows that \(\DUConf_n(\Gamma)\) is path connected if and only if \(\Gamma\) is connected and has more than \(n\) vertices.
However, the necessary conditions for connectivity of \(\DConf_n(\Gamma)\) are more complex (see Theorem 2.9 of \cite{abrams2000configurationspaces}).
We list some instances of \(\DConf_n(\Gamma)\) that are connected below.
\begin{prop}
    \label{prop:connected_dconf}
    The following spaces are connected.
\begin{enumerate}
\item \(\DConf_n(K_5)\) when \(n \in \{2,3\}\)
\item \(\DConf_n(K_{3,3})\) when \(n \in \{2,4\}\)
\item \(\DConf_3(K_{2,4})\)
\end{enumerate}
\end{prop}
\begin{proof}
See Theorem 2.9 of \cite{abrams2000configurationspaces}.
\end{proof}

\subsection{Matchings and Euler characteristics}
\label{subsection:euler}
To compute the Euler characteristic of \(\DConf_n(\Gamma)\) we need to count the number of \(k\)-cubes present in \(\DConf_n(\Gamma)\).  
Observe thata \(k\)-cube exists in \(\DConf_n(\Gamma)\) if and only if 
\(k\) particles are free to move simultaneously while \(n - k\) particles remain fixed.
Since \(k\) particles are only able to move simultaneously when they are placed on \(k\)
disjoint edges, there is a one-to-one correspondence between \(k\)-\textit{matchings} and \(k\)-cubes in \(\DConf_n(\Gamma)\).

\begin{defn}
    \label{defn:matching-set}
\(\mathcal{M}_k(\Gamma)\) is the collection consisting of all \(k\)-matchings in \(\Gamma\). That is, it is the collection of all unordered sets of exactly \(k\) disjoint edges in \(\Gamma\)
where \(\mathcal{M}_0(\Gamma) = \{\emptyset\}\).
\end{defn}

Note that \(\abs{\mathcal{M}_0(\Gamma)} = 1\).
By ignoring the order we place down our particles in, 
we can express the number of \(k\)-cubes in \(\DUConf_n(\Gamma)\) as
\[
    \abs{\mathcal{M}_k(\Gamma)} \cdot\binom{\abs{V(\Gamma)} - 2k}{n - k}.
\]
We now express the Euler characteristic of \(\DUConf_n(\Gamma)\) in the following lemma.
\begin{lem}
\label{lem:eulercharacteristic}
\[
\chi(\DUConf_n(\Gamma)) = \sum_{k=0}^{n} (-1)^k \abs{\mathcal{M}_k(\Gamma)} \binom{\abs{V(\Gamma)} - 2k}{n - k}
\]
\end{lem}
We start with \(\DUConf_n(\Gamma)\) since it is generally easier to count in compared to \(\DConf_n(\Gamma)\).
Furthermore, counting the number of \(k\)-matchings for small graphs can be done efficiently 
with a simple sweep over all \(k\) combinations of disjoint edges in \(\Gamma\).
In Figure \ref{fig:euler_characteristics}, we have computed the Euler characteristic for most of the graphs in Corollary \ref{cor:manifold-classification} (\(\DConf_3(K_5\sqcup K_1)\) is just the disjoint union of \(\DConf_2(K_5)\) and \(\DConf_3(K_5)\)).
When \(\DConf_n(\Gamma)\) is connected, it is an \(n!\)-fold cover of \(\DUConf_n(\Gamma)\),
meaning \(\chi(\DConf_n(\Gamma)) = n!\chi(\DUConf_n(\Gamma))\). Using Proposition \ref{prop:connected_dconf}, we also add \(\chi(\Conf_n(\Gamma))\) to Figure \ref{fig:euler_characteristics}.

\begin{figure}[h!]
\centering
\begin{tabular}{c | c | c | c}
    \(\Gamma\) & \(n\) & \(\chi(\DUConf_n(\Gamma))\) & \(\chi(\Conf_n(\Gamma))\)\\
    \hline
    \(K_5\) & 2 & -5 & -10\\
    \(K_{3,3}\) & 2 & -3 & -6\\
    \(K_5\) & 3 & -5 & -30\\
    \(K_{2,4}\) & 3 & -4 & -24\\
    \(K_{3,3}\) & 4 & -3 & -72 
\end{tabular}
\caption{Euler characteristic of certain discretized graph configuration spaces.}
\label{fig:euler_characteristics}
\end{figure}

    \section{Pseudomanifolds} \label{section:pseudomanifolds}
In this section, we answer Question \ref{qst:graph-manifolds} by determining
exactly which discretized graph configuration spaces are \(m\)-dimensional pseudomanifolds (recall Definition \ref{defn:fernandes-pseudomanifold}) in Theorem \ref{thm:m-pseudomanifolds}, then
pointing out when these pseudomanifolds fail to be manifolds in Propositions \ref{prop:abrams-non-manifolds}, \ref{prop:remaining-non-manifolds1}, and \ref{prop:remaining-non-manifolds2}. Corollary \ref{cor:manifold-classification} then checks that the remaining spaces are indeed manifolds.

The key idea we use is stated in Lemma \ref{lem:special-edges}.
Essentially, if we want to know if a graph configuration space is an \(m\)-pseudomanifold,
then any \((m-1)\)-cell in our configuration space must be a face of exactly two \(m\)-cells.
In the language of placing particles on graphs, this roughly translates to
``for any collection of \((m-1)\) moving particles, the remaining \(n - (m-1)\) particles
placed on the graph should be able to move in a limited way.''
Our approach uses this idea translated into the language of matchings.

Note that if \(\mathcal{V}\) is a set of vertices in a graph \(\Gamma\), then the graph \(\Gamma - \mathcal{V}\)
refers to the subgraph of \(\Gamma\) with vertex set \(V(\Gamma) \setminus \mathcal{V}\)
and edge set \(\{e \in E(\Gamma) \mid \text{neither endpoint of \(e\) is in \(\mathcal{V}\)}\}\).
Also, if \(\mathcal{E}\) is a set of edges in \(\Gamma\), then the graph \(\Gamma - \mathcal{E}\)
refers to the subgraph of \(\Gamma\) with vertex set \(V(\Gamma)\)
and edge set \(E(\Gamma)\setminus\mathcal{E}\).
Figure \ref{fig:K5-minus-K2}
illustrates this notation.

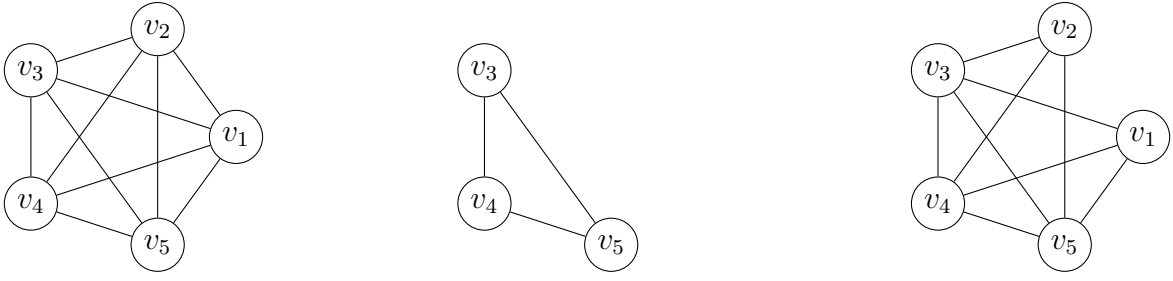
\begin{figure}[h!]
    \centering
    \begin{tikzpicture}
        \begin{scope}[shift={(-3, 0)}]
        \def\radius{1.5}
        \foreach \i in {1,...,5} {
            \node[dormant_graph_node] (v\i) at ({(\i-1)*360/5}:\radius) {\(v_{\i}\)};
        }
        \foreach \i in {1,...,5} {
            \foreach \j in {\i,...,5} {
                \ifnum \i=\j
                \else
                    \draw (v\i) -- (v\j);
                \fi
            }
        } 
        \end{scope}

        \begin{scope}[shift={(3, 0)}]
        \def\radius{1.5}
        \foreach \i in {3,...,5} {
            \node[dormant_graph_node] (v\i) at ({(\i-1)*360/5}:\radius) {\(v_{\i}\)};
        }
        \foreach \i in {3,...,5} {
            \foreach \j in {\i,...,5} {
                \ifnum \i=\j
                \else
                    \draw (v\i) -- (v\j);
                \fi
            }
        } 
        \end{scope}

        \begin{scope}[shift={(9, 0)}]
        \def\radius{1.5}
        \foreach \i in {1,...,5} {
            \node[dormant_graph_node] (v\i) at ({(\i-1)*360/5}:\radius) {\(v_{\i}\)};
        }
        \foreach \i in {1,...,5} {
            \foreach \j in {\i,...,5} {
                \ifnum \i=\j
                \else
                    \ifnum \i=1
                        \ifnum \j=2
                        \else
                        \draw (v\i) -- (v\j);
                        \fi
                    \else
                        \draw (v\i) -- (v\j);
                    \fi
                \fi
            }
        } 
        \end{scope}

    \end{tikzpicture}
    \caption{\(\Gamma\) versus \(\Gamma - \{v_1, v_2\}\) versus \(\Gamma - \{v_1 v_2\}\)}
    \label{fig:K5-minus-K2}
\end{figure}

Before Definition \ref{defn:matching-set}, we remarked that the movement of \(k\) particles
on a graph \(\Gamma\) corresponds to a particular \(k\)-matching in \(\Gamma\).
So, if \((m-1)\) particles are moving, the remaining \(n - (m - 1)\) particles
can only move on the part of the graph away from the corresponding \((m-1)\)-matching \(M\).
That is,
if \(\DConf_n(\Gamma)\) is an \(m\)-pseudomanifold, then
for any \((m-1)\)-matching \(M\) in \(\Gamma\),
any collection \(\mathcal{V}\) of \(n - (m-1)\) particles placed on \(\Gamma - V(M)\) should be able to move in a limited way.

Taking this idea to its logical end, we completely characterize 
what \(\Gamma - V(M)\) can be in Lemmas 
\ref{lem:manifold-with-claw},
\ref{lem:manifold-with-cycle-and-isolated},
and \ref{lem:manifold-with-only-cycles}.
Once we know what subgraphs of \(\Gamma\) are possible, we show that no two of these subgraphs can
occur as \(\Gamma - V(M)\) simultaneously for different matchings in Theorem \ref{thm:exclusive-subgraphs}.
Knowing that \(\Gamma - V(M) \cong \Lambda\) for some fixed subgraph \(\Lambda\)
and any \((m-1)\)-matching \(M\), we can reconstruct what the original graph \(\Gamma\) must have been. Due to their length, we postpone these proofs until Section \ref{section:special-subgraphs}. 
These reconstruction arguments put the last puzzle piece in place necessary
to classify the pseudomanifolds in Theorem \ref{thm:m-pseudomanifolds}.

%
% ----- Input subsections -----
%

\subsection{Key Lemmas}
The first thing to observe about \(n\) and \(\Gamma\) after assuming \(\DConf_n(\Gamma)\) is an \(m\)-pseudomanifold is that there needs to be enough space in \(\Gamma\) for particles to move around.

\begin{lem}
\label{lem:graph-size}
    If \(\DConf_n(\Gamma)\) is a closed \(m\)-pseudomanifold, then
    \(\Gamma\) has at least \(m + n\) vertices and \(n \ge m\).
    Moreover, \(\Gamma\) contains at least one \(m\)-matching,
    and for any \((m-1)\)-matching \(M\) in \(\Gamma\)
    \[
        1 \le n - (m - 1) \le \abs{V(\Gamma - V(M))} - 1.
    \]
\end{lem}

\begin{proof}
    If \(\Gamma\) has less than \(n\) vertices, then \(\DConf_n(\Gamma)\) is empty;
    so \(\Gamma\) has at least \(n\) vertices.
    If \(\Gamma\) has less than \(n + m\) vertices, then
    for any configuration of \(n\) particles on \(\Gamma\), there are
    at most \(m-1\) destinations for particles to move to. Since exactly \(m\) particles need to be able to move simultaneously for an \(m\)-cell to exist in \(\DConf_n(\Gamma)\), it follows that \(\Gamma\) must have at least \(n + m\) vertices.
    Similarly, if \(n < m\) or \(\Gamma\) contains no \(m\)-matchings, then there are not enough particles that can move simultaneously for an \(m\)-cube to exist in \(\DConf_n(\Gamma)\).

    Finally, suppose that \(M\) is an \((m-1)\)-matching in \(\Gamma\) and observe:
    \[
        \abs{V(\Gamma - V(M))} = \abs{V(\Gamma)} - 2(m-1) \ge (m + n) - 2(m-1) = n - (m - 1) + 1.
    \]
    Since \(n \ge m\), we have \(n - (m - 1) \ge 1\).    
\end{proof}

This next lemma is the key to all following proofs.
It is the matching translation of the \(m\)-pseudomanifold condition (see Definition \ref{defn:fernandes-pseudomanifold}) for \(\DConf_n(\Gamma)\).

\begin{lem}
    \label{lem:special-edges}
    \(\DConf_n(\Gamma)\) is an \(m\)-pseudomanifold if and only if
    for any collection \(\mathcal{V}\) of \(n - (m - 1)\) vertices in \(\Gamma - V(M)\), there exists exactly two
    edges \(e_1\) and \(e_2\) in \(\Gamma - V(M)\) that are incident to both vertices in \(\mathcal{V}\)
    and vertices not in \(\mathcal{V}\).

    Furthermore, these edges \(e_1\) and \(e_2\) satisfy all of the following conditions.
    \begin{enumerate}[label=(\roman*)]
    \item \(e_1\) and \(e_2\) are both incident to one vertex in \(\mathcal{V}\)
    \item \(e_1\) and \(e_2\) are both incident to one vertex not in \(\mathcal{V}\)
    \item \(e_1\) and \(e_2\) share exactly one common vertex.
    \end{enumerate}

    Equivalently, exactly one of the following holds in Figure \ref{fig:lem:special-edges}
    \begin{figure}[h!]
        \centering
        \begin{enumerate*}[label=(\arabic*)]
            \item \label{fig:lem:manifolds_1_1}
            \begin{minipage}{.3\textwidth}
                \centering
                \(v_1 \in \mathcal{V}\) \textit{and} \(w_1, w_2 \not \in \mathcal{V}\) \\
                \vspace{1em}
                \begin{tikzpicture}
                    \node (v1) at (3, 2) [circle, draw] {\(v_1\)};
                    \node (w1) at (2, 0) [circle, draw] {\(w_1\)};
                    \node (w2) at (4, 0) [circle, draw] {\(w_2\)};

                    \draw (v1) -- (w1) node[midway, left] {\(e_1\)};
                    \draw (v1) -- (w2) node[midway, right] {\(e_2\)};
                \end{tikzpicture} 
            \end{minipage}

            \hspace{3em}

            \item \label{fig:lem:manifolds_1_2}
            \begin{minipage}{.3\textwidth}
                \centering
                \(v_1, v_2 \in \mathcal{V}\) \textit{and} \(w_1 \not \in \mathcal{V}\) \\
                \vspace{1em}
                \begin{tikzpicture}
                    \node (v1) at (2, 2) [circle, draw] {\(v_1\)};
                    \node (v2) at (4, 2) [circle, draw] {\(v_2\)};
                    \node (w1) at (3, 0) [circle, draw] {\(w_1\)};

                    \draw (v1) -- (w1) node[midway, left] {\(e_1\)};
                    \draw (v2) -- (w1) node[midway, right] {\(e_2\)};
                \end{tikzpicture}
            \end{minipage}
        \end{enumerate*}
        \caption{Lemma \ref{lem:special-edges} possibilities.}
        \label{fig:lem:special-edges}
    \end{figure}
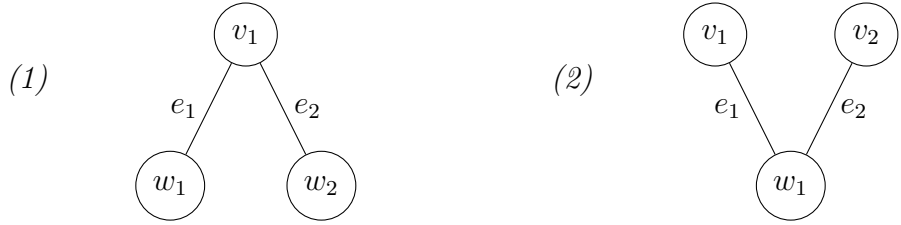
\end{lem}

\begin{proof}
    We first show the ``only if'' part of Lemma's statement.
    Suppose \(\DConf_n(\Gamma)\) is an \(m\)-pseudomanifold.
    Let \(\mathcal{V}\) be a collection of \(n - (m - 1)\) vertices in \(\Gamma - V(M)\)
    and \(p\) be the configuration of \(n\) particles placed on \(\Gamma\) such that
    \(n - (m - 1)\) particles are placed at each vertex in \(\mathcal{V}\)
    and \(m - 1\) particles on each edge in \(M\).
    
    As the particles moves along the edges in \(M\), an \((m-1)\)-cube containing \(p\)
    is spanned in \(\DConf_n(\Gamma)\).
    Since \(\DConf_n(\Gamma)\) is an \(m\)-pseudomanifold,
    this \((m-1)\)-cube must border exactly two distinct \(m\)-cubes.

    For this \((m-1)\)-cube to border two \(m\)-cubes,
    there needs to exist \textbf{two} additional mutually exclusive particle movements.
    Let \(v_1, v_2 \in \mathcal{V}\) and \(w_1, w_2 \not \in \mathcal{V}\) be the vertices corresponding
    to the origins and destinations of these additional particle movements respectively.
    If \(v_1 \neq v_2\) and \(w_1 \neq w_2\), then two particles would be able to move simultaneously
    resulting in an \((m+1)\)-cube being spanned in \(\DConf_n(\Gamma)\) contradicting 
    that it is an \(m\)-pseudomanifold.
    Since \(\Gamma\) is simple, it cannot be that \(v_1 = v_2\) and \(w_1 = w_2\).

    To see the ``if'' part of the statement. Notice that if \((m-1)\) particles are traveling along
    an \((m-1)\)-matching \(M\), then we have an \((m-1)\)-cube in \(\DConf_n(\Gamma)\).
    By restricting the remaining \(n - (m-1)\) particles to move in exactly two ways,
    there are exactly two \(m\)-cubes sharing this \((m-1)\)-cube as a face.
    Since \(M\) and the remaining \(n - (m-1)\) particles were arbitrary,
    \(\DConf_n(\Gamma)\) must be an \(m\)-pseudomanifold.
\end{proof}

From this point forward in this section we assume that \(\DConf_n(\Gamma)\) is a closed \(m\)-pseudomanifold
and \(M\) is an \((m-1)\)-matching in \(\Gamma\).
Note that if \(m = 1\), then \(\Gamma - V(M) = \Gamma\).
The next few lemmas provide significant restrictions
for what \(\Gamma - V(M)\) can be.

\begin{lem}
    \label{lem:manifold-max-degree}
    Every vertex in \(\Gamma - V(M)\) has degree at most \(3\).
\end{lem}

\begin{proof}
    Suppose \(\Gamma - V(M)\) has a vertex \(v\) with degree greater than \(3\).
    We proceed by cases on \(n\).

    \textbf{Case 1:} \(n - (m - 1) \le \deg(v)\).
    In the context of Lemma \ref{lem:special-edges},
    construct \(\mathcal{V}\) so that all \(n - (m - 1)\)
    vertices are adjacent to \(v\).
    Then, there are too many edges with one endpoint in \(\mathcal{V}\)
    and the other outside of \(\mathcal{V}\) for Lemma \ref{lem:special-edges} to hold.

    \textbf{Case 2:} \(n - (m - 1) > \deg(v)\).
    Lemma \ref{lem:graph-size} guarantees that \(\Gamma - V(M)\) has at least
    \(n - m + 2 > \deg(v) + 1\) vertices.
    Let \(\mathcal{V}\) be a set containing \(v\) and every vertex adjacent to \(v\).
    Lemma \ref{lem:special-edges} guarantees there is at least one vertex \(w \not \in \mathcal{V}\).
    Let \(\mathcal{V}' = \left(\mathcal{V}\setminus\{v\}\right)\cup\{w\}\).
    Similarly to the previous case, there are too many edges with one endpoint in and outside of \(\mathcal{V}'\) for Lemma \ref{lem:special-edges} to hold.
\end{proof}

\begin{lem}
\label{lem:max-subgraph-matching}
If \(n > m\) and \(\Gamma - V(M)\) contains two or more disjoint edges, then \(\Gamma - V(M)\) has exactly \(n - (m - 1) + 1\) vertices.
\end{lem}
\begin{proof}
    Suppose \(v_1 v_2\) and \(v_3 v_4\) are two disjoint
    edges in \(\Gamma - V(M)\). By
    Lemma \ref{lem:graph-size}, 
    \(\Gamma - V(M)\) has at least \(n - (m - 1) + 1\) vertices, so suppose for the sake of contradiction that  \(\Gamma - V(M)\) has \(n - (m - 1) + 2\) or more vertices.

    Let \(\mathcal{V}\) be a collection of \(n - (m - 1)\)
    vertices in \(\Gamma - V(M)\) including \(v_1\)
    and \(v_3\), but excluding \(v_2\) and \(v_4\).
    Then, the edges \(v_1 v_2\) and \(v_3 v_4\)
    both have endpoints inside and outside of \(\mathcal{V}\),
    yet these edges do not share a common endpoint contradicting Lemma \ref{lem:special-edges}.
\end{proof}

\begin{lem}
\label{lem:easy-only-cycles}
If \(n - (m - 1) \in \{1, \abs{V(\Gamma - V(M))} - 1\}\),
then every vertex in \(\Gamma - V(M)\) has degree 2.
\end{lem}

\begin{proof}
Let \(v\) be some vertex in \(\Gamma - V(M)\).
If \(n - (m - 1) = 1\),  let \(\mathcal{V} = \{v\}\),
otherwise if \(n - (m - 1) = \abs{V(\Gamma - V(M))} - 1\),
let \(\mathcal{V} = V(\Gamma - V(M))\setminus\{v\}\).

In either case, Lemma \ref{lem:special-edges} guarantees that
there must be exactly two edges in \(\Gamma - V(M)\) with endpoints inside and
outside of \(\mathcal{V}\). Since these two edges
will be the only edges incident to \(v\), and since \(v\)
was arbitrary, every vertex in \(\Gamma - V(M)\) has degree 2.
\end{proof}

\subsection{Possible Subgraphs}
We now have the tools necessary to determine what graph \(\Gamma - V(M)\) can be. Lemmas \ref{lem:manifold-with-claw}, \ref{lem:manifold-with-cycle-and-isolated}, and \ref{lem:manifold-with-only-cycles} below do this classification.

\begin{lem}
    \label{lem:manifold-with-claw}
    If \(\Gamma - V(M)\) contains a degree \(3\) vertex,
    then \(\Gamma - V(M) \cong K_{1,3}\) and \(n = m + 1\).
\end{lem}

\begin{proof}
    Suppose \(\Gamma - V(M)\) contains a degree \(3\) vertex \(v\).
    Observe that \(n = m + 1\) if and only if \(n - (m - 1) = 2\).
    We first show that \(n\) must be equal to \(m + 1\) by considering two cases.

    \textbf{Case 1:} \(n - (m - 1) > 2\).
    If \(n - (m - 1) = 3\), then construct \(\mathcal{V}\) to include just the neighbors of \(v\).
    Then, there are too many edges with one endpoint inside and one outside of \(\mathcal{V}\)
    for Lemma \ref{lem:special-edges} to hold.
    Otherwise, if \(n - (m-1) > 3\), then let \(\mathcal{V}\) be a collection of \(n - (m - 1)\) vertices
    that includes \(v\) and its three neighbors.
    By Lemma \ref{lem:special-edges}, there must be at least one additional vertex
    \(w\) outside of \(\mathcal{V}\) connected to one of the neighbors of \(v\).
    Let \(\mathcal{V}' = (\mathcal{V}\setminus\{v\})\cup\{w\}\).
    Again, there are too many edges with one endpoint inside and one outside of \(\mathcal{V}'\)
    for Lemma \ref{lem:special-edges} to hold.

    \textbf{Case 2:} \(n - (m - 1) < 2\).
    Since Lemma \ref{lem:graph-size} guarantees
    \(n \ge m\), we must have that \(n - (m - 1) = 1\).
    Since \(v\) is degree 3, we contradict Lemma \ref{lem:easy-only-cycles}.

    Next, we show that there cannot be any other vertices in \(\Gamma - V(M)\) besides \(v\) and its neighbors.
    Suppose there was a vertex \(w \neq v\) in \(\Gamma - V(M)\) that is not adjacent to \(v\).
    By choosing \(\mathcal{V} = \{v, w\}\),
    the three neighbors of \(v\) again contradict Lemma \ref{lem:special-edges}.

    Finally, we show that there cannot be any edges in \(\Gamma - V(M)\)
    besides the three edges incident to \(v\).
    Suppose there was another edge \(e\) in \(\Gamma - V(M)\) not incident to \(v\).
    Let \(a\) and \(b\) be the vertices incident to \(e\) and \(c\) be the other vertex adjacent to \(v\)
    (see Figure \ref{fig:lem:manifold-with-claw:1}).
    By choosing \(\mathcal{V} = \{a, c\}\), the edges \(cv\), \(av\), and \(ab\) contradict Lemma \ref{lem:special-edges}.

    \begin{figure}[h!]
        \centering
        \begin{tikzpicture}
            \node (v) at (0,0) [circle, draw] {\(v\)};
            \node (a) at (-1,-1) [circle, draw] {\(a\)};
            \node (b) at (1,-1) [circle, draw] {\(b\)};
            \node (c) at (0,1.4) [circle, draw] {\(c\)};

            \draw (v) -- (a);
            \draw (v) -- (b);
            \draw (v) -- (c);
            \draw (a) -- (b) node[midway, below] {\(e\)};
        \end{tikzpicture}
        \caption{\(\Gamma - V(M)\) with an extra edge \(e\)}
        \label{fig:lem:manifold-with-claw:1}
    \end{figure}
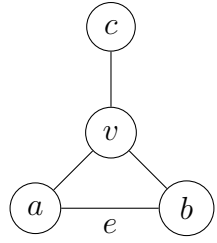

    Therefore \(\Gamma - V(M)\) must consist solely of \(v\), its neighbors, and the edges incident to \(v\)
    i.e.\ \(\Gamma - V(M) \cong K_{1,3}\).
\end{proof}

\begin{lem}
\label{lem:manifold-with-cycle-and-isolated}
    If \(\Gamma - V(M)\) contains no degree 3 vertices and
    \(n - (m - 1) \not \in \{1, \abs{V(\Gamma - V(M))} - 1\}\),
    then \(\Gamma - V(M) \cong K_3 \sqcup K_1\)
    and \(n = m + 1\).
\end{lem}

\begin{proof}
Since \(\Gamma - V(M)\) contains 
no degree 3 vertices, Lemma \ref{lem:manifold-max-degree} guarantees that
every vertex in \(\Gamma - V(M)\) has degree at most \(2\).

Suppose for the sake of contradiction that \(\Gamma - V(M)\)
contains no cycles. Then, \(\Gamma - V(M)\) is a collection
of disjoint paths \(\Lambda_1, \Lambda_2, \cdots, \Lambda_r\).
Aiming to use Lemma \ref{lem:special-edges}, construct \(\mathcal{V}\) so that it consists of vertices chosen one
by one from one end of the path \(\Lambda_1\) all the way to the other end. Until \(n - (m - 1)\) vertices have been selected,
move on to \(\Lambda_2\), \(\Lambda_3\), etc.
Once \(n - (m - 1)\) vertices have been selected,
observe that Lemma \ref{lem:special-edges} is violated since there is at most one edge with one endpoint
inside and one outside of \(\mathcal{V}\).
Hence \(\Gamma - V(M)\) must contain a cycle \(C\).

Note that \(n > m\) under the hypothesis of this
Lemma otherwise \(1 < n - (n - 1) = 1\).
Notice that if \(C\) has more than 3 edges,
or there is another edge in \(\Gamma - V(M)\)
other than the ones on \(C\), then
\(\Gamma - V(M)\) has two or more disjoint edges.
However, since \(n > m\), Lemma \ref{lem:max-subgraph-matching} says that this only happens when 
\(\Gamma - V(M)\) has exactly \(n - (m - 1) + 1\) vertices.
Hence \(C \cong K_3\) and there are no other edges in
\(\Gamma - V(M)\).

Now, \(n\) cannot be larger than \(m + 1\) otherwise
letting \(\mathcal{V}\) contain all the vertices
in \(C\), there are no edges in \(\Gamma - V(M)\)
with endpoints inside and outside of \(\mathcal{V}\).
Hence \(n = m + 1\).

Finally, since \(\Gamma - V(M)\) contains
no edges other than those in \(C\), \(\Gamma - V(M)\)
must consist of \(C\) and potentially some number \(q\) of isolated vertices. 
If \(q = 0\), then \(\Gamma - V(M) = C\) 
contains exactly 3 vertices.
However, this means 
\(2 = n - (m - 1) < \abs{V(\Gamma - V(M))} - 1 = 2\).
If \(q > 1\), then we can choose \(\mathcal{V}\)
so that it consists of two isolated vertices
contradicting Lemma \ref{lem:special-edges}.
Hence \(q = 1\) i.e.\ \(\Gamma - V(M) \cong K_3 \sqcup K_1\).
\end{proof}

\begin{lem}
\label{lem:manifold-with-only-cycles}
If \(\Gamma - V(M)\) contains no degree 3 vertices,
and \(n - (m - 1) \in \{1, \abs{V(\Gamma - V(M))} - 1\}\),
then exactly one of the following is true.
\begin{enumerate}
\item \(m = 1\), \(\Gamma\) is a disjoint collection of cycles, and \(n \in \{1, \abs{V(\Gamma)} - 1\}\)
\item \(m > 1\), \(\Gamma - V(M) \cong K_3\), and \(n \in \{m, m+1\}\).
\item \(m > 1\), \(\Gamma - V(M) \cong K_{2,2}\), and \(n \in \{m, m+2\}\).
\end{enumerate}
\end{lem}

\begin{proof}
Lemma \ref{lem:easy-only-cycles} guarantees that
\(\Gamma - V(M)\) is a disjoint collection of cycles.
The first possible conclusion follows quickly from
observing that if \(m = 1\), then \(n - (m - 1) = n\) and  \(\Gamma - V(M) = \Gamma\).
To see that either the second or third conclusion is true, first suppose \(m > 1\).
Notice that any choice of a set \(\mathcal{V}\) of \(n - (m - 1)\) vertices will satisfy the hypothesis of Lemma \ref{lem:special-edges}. However, when we change
which \((m-1)\)-matching we are using, we will see
Lemma \ref{lem:special-edges} fail.
We will use this technique to show that \(\Gamma - V(M)\)
is a cycle with at most 4 vertices.
Notice that if our claim is true, then
\(\Gamma - V(M)\) is either a 3-cycle,
or a 4-cycle. Since \(n - (m - 1) \in \{1, \abs{V(\Gamma - V(M))} - 1\}\),
if \(\Gamma - V(M)\) is a 3-cycle,
then \(n \in \{m, m + 1\}\),
and if \(\Gamma - V(M)\) is a 4-cycle,
then \(n \in \{m, m+2\}\).
We now prove our claim in two steps.

\textbf{Step 1:} \(\Gamma - V(M)\) consists of exactly one cycle.

Suppose \(\Gamma - V(M)\) has two cycles \(C\) and \(D\).
Since \(\Gamma\) is simple, \(\Gamma\) can have two cycles only if \(C\) and \(D\) have length at least \(3\). Hence the number of vertices in \(\Gamma - V(M)\) must be at least \(6\).
    Since \(m > 1\), \(M\) contains at least one edge \(e_1\).
    Let \(e_2\) and \(e_3\) be edges on \(C\) and \(D\) respectively.

    \begin{figure}[h!]
        \centering
        \begin{tikzpicture}
            \begin{scope}[shift={(-2,0)}]
                \node (v1) at (0,0) [dormant_graph_node] {};
                \node (v2) at (0,1.5) [dormant_graph_node] {};
                \draw (v1) -- (v2) [dashed] node[midway, right] {\(e_1\)};
            \end{scope}
            \begin{scope}
                \node (c0) at ({1.5 * cos(60) + 1.5 * cos(30)},{-1.5 * sin(60) - 1.5 * sin(30)}) {};
                \node (c1) at ({1.5 * cos(60)},{-1.5 * sin(60)}) [active_graph_node] {};
                \node (c2) at (0,0) [active_graph_node] {};
                \node (c3) at (0,1.5) [active_graph_node] {};
                \node (c4) at ({1.5 * cos(60)},{1.5 + 1.5 * sin(60)}) [active_graph_node] {};
                \node (c5) at ({1.5 * cos(60) + 1.5 * cos(30)},{1.5 + 1.5 * sin(60) + 1.5 * sin(30)}) {};
                \draw (c0) -- (c1);
                \draw (c1) -- (c2);
                \draw (c2) -- (c3) node [midway, right] {\(e_2\)};
                \draw (c3) -- (c4);
                \draw (c4) -- (c5);
            \end{scope}

            \begin{scope}[shift={(3,0)}]
                \node (d0) at ({1.5 * cos(60)},{-1.5 * sin(60)}) {};
                \node (d1) at (0, 0) [active_graph_node] {};
                \node (d2) at (0, 1.5) [active_graph_node] {};
                \node (d3) at ({1.5 * cos(60)},{1.5 + 1.5 * sin(60)}) {};
                \draw (d0) -- (d1);
                \draw (d1) -- (d2) node [midway, right] {\(e_3\)};
                \draw (d2) -- (d3);

            \end{scope}
        \end{tikzpicture}
    \caption{Parts of cycles \(C\) and \(D\) in \(\Gamma - V(M)\)}
    \label{fig:lem:manifold-one-cycle:1}
    \end{figure}
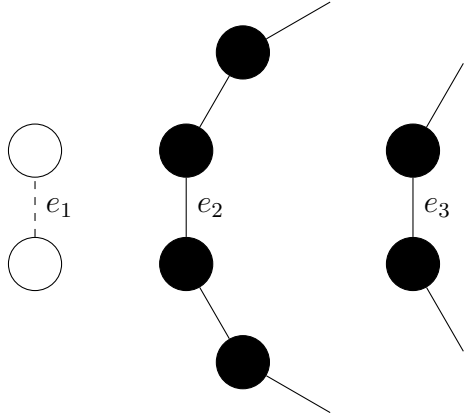

    Define \(M' = (M \setminus \{e_1\}) \cup \{e_2\}\) and \(M'' = (M \setminus \{e_1\}) \cup \{e_3\}\).
    Notice \(\Gamma - V(M')\) and \(\Gamma - V(M'')\) cannot have any degree 3 vertices otherwise
    Lemma \ref{lem:manifold-with-claw} would imply
    these subgraphs both have only 4 vertices. So,
    Lemma \ref{lem:easy-only-cycles} guarantees
    \(\Gamma - V(M')\) and \(\Gamma - V(M'')\)
    are also collections of disjoint cycles.
    
    So, \(e_1\) must belong to a cycle in \(\Gamma - V(M')\).
    Observe that the only way for this to happen is if \(e_1\) belongs to the ``broken cycle'' \(C - V(\{e_2\})\)
    in \(\Gamma - V(M')\) (see Figure \ref{fig:lem:manifold-one-cycle:2}).

    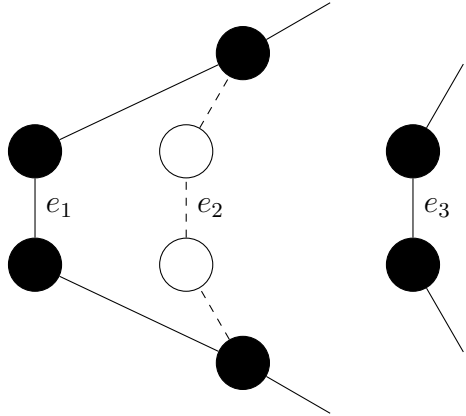
\begin{figure}[h!]
        \centering
        \begin{tikzpicture}
            \begin{scope}[shift={(-2,0)}]
                \node (v1) at (0,0) [active_graph_node] {};
                \node (v2) at (0,1.5) [active_graph_node] {};
                \draw (v1) -- (v2) node[midway, right] {\(e_1\)};
            \end{scope}
            \begin{scope}
                \node (c0) at ({1.5 * cos(60) + 1.5 * cos(30)},{-1.5 * sin(60) - 1.5 * sin(30)}) {};
                \node (c1) at ({1.5 * cos(60)},{-1.5 * sin(60)}) [active_graph_node] {};
                \node (c2) at (0,0) [dormant_graph_node] {};
                \node (c3) at (0,1.5) [dormant_graph_node] {};
                \node (c4) at ({1.5 * cos(60)},{1.5 + 1.5 * sin(60)}) [active_graph_node] {};
                \node (c5) at ({1.5 * cos(60) + 1.5 * cos(30)},{1.5 + 1.5 * sin(60) + 1.5 * sin(30)}) {};
                \draw (c0) -- (c1);
                \draw (c1) -- (c2) [dashed];
                \draw (c2) -- (c3) [dashed] node [midway, right] {\(e_2\)};
                \draw (c3) -- (c4) [dashed];
                \draw (c4) -- (c5);
            \end{scope}

            \draw (v1) -- (c1);
            \draw (v2) -- (c4);
            \begin{scope}[shift={(3,0)}]
                \node (d0) at ({1.5 * cos(60)},{-1.5 * sin(60)}) {};
                \node (d1) at (0, 0) [active_graph_node] {};
                \node (d2) at (0, 1.5) [active_graph_node] {};
                \node (d3) at ({1.5 * cos(60)},{1.5 + 1.5 * sin(60)}) {};
                \draw (d0) -- (d1);
                \draw (d1) -- (d2) node [midway, right] {\(e_3\)};
                \draw (d2) -- (d3);

            \end{scope}
        \end{tikzpicture}
    \caption{\(e_1\) attached to \(C\) in \(\Gamma - V(M')\)}
    \label{fig:lem:manifold-one-cycle:2}
    \end{figure}

    But then, a vertex on \(C\) has degree \(3\) in \(\Gamma - V(M'')\) (if \(C\) has length greater than \(3\))
    or degree \(4\) (if \(C\) has length equal to \(3\)), a contradiction.

\textbf{Step 2:} \(\Gamma - V(M)\) is a cycle with at most 4 vertices.

The previous step shows that \(\Gamma - V(M)\)
consists of a single cycle. Suppose for the sake of contradiction that this
cycle has more than 4 vertices and 
let \(v_1\), \(v_2\), \(v_3\), \(v_4\), and \(v_5\)
be five vertices in a path on it.

Since \(m > 1\), \(M\) contains at least one edge \(w_1 w_2\).
Define \(M' = (M \setminus \{w_1 w_2\}) \cup \{v_2 v_3\}\).
Since \(\Gamma - V(M)\) contains more than \(4\) vertices, \(\Gamma - V(M')\) also contains more than \(4\) vertices. So, \(\Gamma - V(M')\) cannot
have any degree 3 vertices (otherwise we contradict Lemma \ref{lem:manifold-with-claw}), and
Lemma \ref{lem:easy-only-cycles} implies \(\Gamma - V(M')\) is a collection of cycles.

Using the same argument as in step 1
one can conclude that \(\Gamma - V(M')\) must consist
of exactly one cycle. So, \(w_1\) and \(w_2\) must be adjacent to \(v_1\) or \(v_4\) in \(\Gamma - V(M')\) since those vertices
are the only vertices in \(\Gamma - V(M')\) that could have degree less than \(2\).

    \begin{figure}[h!]
        \centering
        \begin{tikzpicture}
            \foreach \i in {1,...,5} \node (v\i) at ({90-72*\i}:1.2) [circle, draw] {\(v_{\i}\)};
            \draw (v1) -- (v2) -- (v3) -- (v4) -- (v5);

            \draw[dotted] (v5) -- (v1);

            \node (w1) at (-3,.7) [circle, draw] {\(w_1\)};
            \node (w2) at (-3,-.7) [circle, draw] {\(w_2\)};
            \draw (w1) -- (w2);

            \draw (w1) -- (v4);
        \end{tikzpicture}
        \caption{\(w_1\) is adjacent to \(v_4\) in \(\Gamma - V(M')\)}
        \label{fig:lem:manifold-cycle-length-1:1}
    \end{figure}
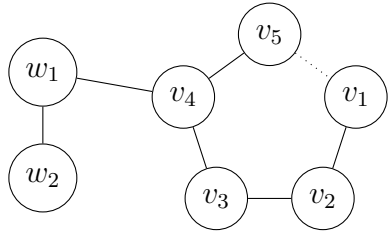

    Without loss of generality, suppose \(w_1\) is adjacent to \(v_4\) in \(\Gamma - V(M')\) (see Figure \ref{fig:lem:manifold-cycle-length-1:1})
    and define \(M'' = (M \setminus \{w_1 w_2\}) \cup \{v_1 v_2\}\).
    Then \(v_4\) has degree \(3\) in \(\Gamma - V(M'')\).
    Lemma \ref{lem:manifold-with-claw} then requires that \(\Gamma - V(M'')\) is just \(v_4\) and its neighbors.
    However, \(\Gamma - V(M'')\) has at least \(5\) vertices. 
\end{proof}

\subsection{Exclusivity of Subgraphs}
Lemmas \ref{lem:manifold-with-claw},
\ref{lem:manifold-with-cycle-and-isolated},
and \ref{lem:manifold-with-only-cycles}
tell us exactly what \(\Gamma - V(M)\)
could possibly be. However, one might ask whether it is possible for two \((m-1)\)-matchings \(M_1 \neq M_2\) to exist in \(\Gamma\) such that \(\Gamma - V(M_1) \not \cong \Gamma - V(M_2)\).
Theorem \ref{thm:exclusive-subgraphs} below explains how this is not possible when \(\DConf_n(\Gamma)\)
is an \(m\)-psuedomanifold.
Note that if the assumption \(m > 1\) or \(n \not\in \{1, \abs{V(\Gamma)} - 1\}\) does not hold, then Lemma \ref{lem:easy-only-cycles} guarantees \(\Gamma\) is the disjoint union of one or more cycle graphs.

\begin{thm}
    \label{thm:exclusive-subgraphs}
    Suppose that \(m > 1\) or \(n \not \in \{1, \abs{V(\Gamma)} - 1\}\).
    \(\DConf_n(\Gamma)\) is a closed \(m\)-psuedomanifold
    if and only if there exists a graph \(\Lambda \in \{K_3, K_{2,2}, K_3 \sqcup K_1, K_{1,3}\}\) such that for any \((m-1)\)-matching \(M\) we have that
    \(\Gamma - V(M) \cong \Lambda\), and
    \begin{enumerate}
        \item If \(\Lambda = K_3\), then \(n \in \{m, m+1\}\).
        \item If \(\Lambda = K_{2,2}\), then \(n \in \{m, m+2\}\).
        \item If \(\Lambda \in \{K_{1,3}, K_3\sqcup K_1\}\), then \(n = m+ 1\).
    \end{enumerate}
\end{thm}

\begin{proof}
    \((\Leftarrow)\) For any such \(\Lambda\), 
    one can check that the hypothesis of Lemma \ref{lem:special-edges} is satisfied.

    \((\Rightarrow)\)
    Let \(M_0\) be some \((m-1)\)-matching.
    Since \(m > 1\) or \(n \not \in \{1, \abs{V(\Gamma)} - 1\}\), Lemmas 
    \ref{lem:manifold-with-claw},
    \ref{lem:manifold-with-cycle-and-isolated},
    and \ref{lem:manifold-with-only-cycles} guarantee \(\Gamma - V(M_0)\) is one of \(\{K_3, K_{2,2}, K_3 \sqcup K_1, K_{1,3}\}\). Let \(\Gamma - V(M_0) = \Lambda_0\).
    
    If \(\Lambda_0 \cong K_3\), then \(\Gamma\) must have \(2m + 1\) vertices. Since in every other case for what \(\Lambda_0\) can be \(\Gamma\) must have \(2m + 2\) vertices, if \(\Lambda_0 \cong K_{3}\), then \(\Gamma - V(M) \cong K_3\) for \textit{every} \((m-1)\)-matching \(M\).
    In addition, Lemma \ref{lem:manifold-with-only-cycles} guarantees that \(n \in \{m, m+1\}\).
    
    Notice that if \(\Lambda_0 \cong K_{2,2}\) then Lemma \ref{lem:manifold-with-only-cycles} forces \(n \in \{m, m+2\}\).
    Since Lemmas \ref{lem:manifold-with-claw}
    and \ref{lem:manifold-with-cycle-and-isolated} guarantee \(n\) must be \(m+1\), if \(\Gamma - V(M)\) is one of \(K_{1,3}\) or \(K_3\sqcup K_1\) for some \((m-1)\)-matching \(M\), it follows that if \(\Lambda_0 \cong K_{2,2}\), then we must have that \(\Gamma - V(M) \cong K_{2,2}\) for \textit{every} \((m-1)\)-matching \(M\).

    Since there are only two possibilities remaining for what \(\Lambda_0\) is, we just need to show that it is impossible for \(\Gamma - V(M) \not \cong \Lambda_0\) for some \((m-1)\)-matching \(M \neq M_0\).
    The proof of this fact will be split into the next two lemmas.
\end{proof}

Recall from Definition \ref{defn:matching-set} that \(\mathcal{M}_k(\Gamma)\) is the set of all \(k\)-matchings in \(\Gamma\).
We denote the symmetric difference of two sets \(A\) and \(B\) by \(A \triangle B = (A \setminus B) \cup (B \setminus A)\).
\begin{lem}
    \label{lem:matching-connectivity}
    Let \(k\) be a positive integer and suppose \(\Gamma\) is a graph where for any \(M \in \mathcal{M}_k(\Gamma)\),
    the graph \(\Gamma - V(M)\) contains at least one edge.
    Then, for any \(M, M' \in \mathcal{M}_k(\Gamma)\),
    there exists a finite sequence \(M_1, M_2, M_3, \ldots, M_n \in \mathcal{M}_k(\Gamma)\)
    such that \(M_1 = M\), \(M_n = M'\), and
    \(\abs{M_i \triangle M_{i+1}} = 2\) for each \(i = 1, 2, \ldots, n-1\).
\end{lem}

\begin{proof}
    We proceed by induction on \(\abs{M \triangle M'}\).
    Notice that since \(M\) and \(M'\) both have \(k\) elements,
    \(\abs{M \triangle M'}\) is always even.
    If \(\abs{M \triangle M'} = 2\), then \(M, M'\) is the desired sequence of matchings.
    So suppose the lemma holds for any pair of \(k\)-matchings \(M\) and \(M'\) with \(\abs{M \triangle M'} = 2t\)
    for some \(t \ge 0\).
    Let \(M\) and \(M'\) be two \(k\)-matchings with \(\abs{M \triangle M'} = 2(t + 1)\).

    Note that \(\Gamma(M\triangle M')\) refers to the subgraph of \(\Gamma\)
    with edges \(M\triangle M'\) and vertices consisting of those in \(\Gamma\)
    which are an endpoint of any edge in \(M\triangle M'\).
    Since \(\abs{M \triangle M'} > 0\), there exists at least one edge \(e\) in \(M \setminus M'\).
    % This is a standard fact about the symmetric difference of matchings. 
    % In fact, I could cite it from "introduction to graph theory" by Douglas West.
    Observe that for any vertex \(v\) incident to an edge in \(\Gamma(M \triangle M')\), the degree of \(v\) must be at most \(2\)
    since at most one edge in \(M\) and at most one edge in \(M'\) can be incident to \(v\).
    % So, the subgraph formed by the edges in \(M \triangle M'\) is a disjoint union of paths and even cycles
    % where the edges in the paths and cycles alternate between being in \(M\) and \(M'\).

    \textbf{Case 1:} \(e\) is adjacent to at most one edge in \(\Gamma(M \triangle M')\).

    There also exists at least one edge \(e'\) in \(M'\setminus M\).
    If \(e\) is adjacent to an edge in \(\Gamma(M \triangle M')\),
    let \(e'\) be this edge. Otherwise, let \(e'\) be any edge in \(M' \setminus M\).
    Define \(M'' = (M' \setminus \{e'\}) \cup \{e\}\).
    Since \(e\) is not adjacent to any edge in \(\Gamma(M \triangle M')\),
    adding \(e\) to the set \(M'' \setminus \{e'\}\) does not create any adjacency issues.
    So, this is a matching where \(\abs{M \triangle M''} = 2t\), and \(\abs{M'' \triangle M'} = 2\).
    By the inductive hypothesis, there exists a sequence of matchings
    connecting \(M\) to \(M''\) where each pair of sequential matchings
    differ by exactly two edges. Appending \(M'\) to this sequence gives the desired sequence connecting \(M\) to \(M'\).

    \textbf{Case 2:} \(e\) is adjacent to exactly two edges in \(\Gamma(M \triangle M')\).

    Let \(e'_a\) and \(e'_b\) be the edges adjacent to \(e\) in \(\Gamma(M \triangle M')\).
    These edges must be in \(M'\).
    Let \(e'_c\) be an edge in \(\Gamma - V(M')\)
    and define \(M'' = (M' \setminus \{e'_a\}) \cup \{e'_c\}\).
    Then, \(e\) is adjacent to exactly one edge in \(M \triangle M''\) which puts us in Case 1.

\end{proof}

\begin{lem}
    \label{lem:claw-complement-or-claw-not-special}
    Suppose \(\Gamma - V(M)\) is isomorphic to \(K_{1,3}\) or \(K_3 \sqcup K_1\)
    for any \((m-1)\)-matching \(M\) in \(\Gamma\).
    There are no distinct \((m-1)\)-matchings \(M_1\) and \(M_2\) in \(\Gamma\)
    such that \(\Gamma - V(M_1) \cong K_{1,3}\) and \(\Gamma - V(M_2) \cong K_3 \sqcup K_1\).
\end{lem}

\begin{proof}

    \tikzset{
            present/.style={circle, draw, text=white, fill=black, inner sep=0pt, minimum size=19pt},
            absent/.style={circle,draw, inner sep=0pt, minimum size=19pt}
    }

    Assume there exists \((m-1)\)-matchings \(M_1\) and \(M_2\) in \(\Gamma\)
    such that \(\Gamma - V(M_1) \cong K_3 \sqcup K_1\) and \(\Gamma - V(M_2) \cong K_{1,3}\).
    By Lemma \ref{lem:matching-connectivity}, there exists a sequence of \((m-1)\)-matchings
    connecting \(M_1\) to \(M_2\). By following this sequence we can find
    two sequential matchings where after erasing the vertices covered by these matchings,
    we get \(K_3 \sqcup K_1\) for one matching and \(K_{1,3}\) for the other.
    We assume \(M_1\) and \(M_2\) are these two matchings.

    We now make a series of ``swaps and inferences'' with these two matchings.
    Let \(M_1 \triangle M_2 = \{ab, vw\}\).
    %Notice that any edge we swap out of \(M_1\) cannot be incident to any vertex in \(\Gamma - V(M_1)\).
    \begin{figure}[h!]
        \centering
        \begin{tikzpicture}
            \node [present] (a) at (0,0) {\(a\)};
            \node [present] (b) at (0,1) {\(b\)};
            \node [present] (c) at (1,0) {\(c\)};
            \node [present] (d) at (2,1) {\(d\)};
            \node [absent] (v) at (2,0) {\(v\)};
            \node [absent] (w) at (3,0) {\(w\)};
            \draw (a) -- (b);
            \draw (b) -- (c);
            \draw (c) -- (a);
            \draw[dashed] (v) -- (w);
        \end{tikzpicture}
        \caption{\(\Gamma - V(M_1)\)}
        \label{fig:lem:exclusive-subgraphs:1}
    \end{figure}
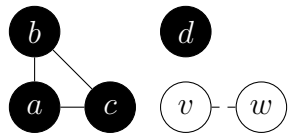
    In Figure \ref{fig:lem:exclusive-subgraphs:1}, we illustrate the known
    edge connections in \(\Gamma - V(M_1)\).
    Swapping \(ab\) for \(vw\) we obtain 
    \(\Gamma - V(M_2)\) which by assumption is \(K_{1,3}\) (see Figure \ref{fig:lem:exclusive-subgraphs:2}).
    
    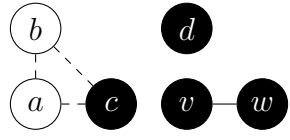
\begin{figure}[h!]
        \centering
        \begin{tikzpicture}
            \node [absent] (a) at (0,0) {\(a\)};
            \node [absent] (b) at (0,1) {\(b\)};
            \node [present] (c) at (1,0) {\(c\)};
            \node [present] (d) at (2,1) {\(d\)};
            \node [present] (v) at (2,0) {\(v\)};
            \node [present] (w) at (3,0) {\(w\)};
            \draw[dashed] (a) -- (b);
            \draw[dashed] (b) -- (c);
            \draw[dashed] (c) -- (a);
            \draw (v) -- (w);
        \end{tikzpicture}
        \caption{\(\Gamma - V((M_1 \setminus \{ab\}) \cup \{vw\})\)}
        \label{fig:lem:exclusive-subgraphs:2}
    \end{figure}

    Without loss of generality, assume \(v\) is the degree \(3\) vertex of \(\Gamma - V(M_2)\).
    We then infer the edge connections in Figure \ref{fig:lem:exclusive-subgraphs:3}.

    \begin{figure}[h!]
        \centering
        \begin{tikzpicture}
            \node [absent] (a) at (0,0) {\(a\)};
            \node [absent] (b) at (0,1) {\(b\)};
            \node [present] (c) at (1,0) {\(c\)};
            \node [present] (d) at (2,1) {\(d\)};
            \node [present] (v) at (2,0) {\(v\)};
            \node [present] (w) at (3,0) {\(w\)};
            \draw[dashed] (a) -- (b);
            \draw[dashed] (b) -- (c);
            \draw[dashed] (c) -- (a);
            \draw (v) -- (w);
            \draw (v) -- (c);
            \draw (v) -- (d);
            
        \end{tikzpicture}
        \caption{\(\Gamma - V(M_2)\)}
        \label{fig:lem:exclusive-subgraphs:3}
    \end{figure}
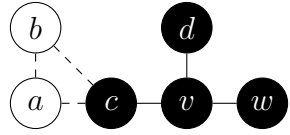

    Next, swap \(cv\) for \(ab\) (see Figure \ref{fig:lem:exclusive-subgraphs:4}).
    There are two possibilities: the resulting subgraph is \(K_{1,3}\) or it is \(K_3 \sqcup K_1\).

    \begin{figure}[h!]
        \centering
        \begin{tikzpicture}
            \node [present] (a) at (0,0) {\(a\)};
            \node [present] (b) at (0,1) {\(b\)};
            \node [absent] (c) at (1,0) {\(c\)};
            \node [present] (d) at (2,1) {\(d\)};
            \node [absent] (v) at (2,0) {\(v\)};
            \node [present] (w) at (3,0) {\(w\)};
            \draw (a) -- (b);
            \draw[dashed] (b) -- (c);
            \draw[dashed] (c) -- (a);
            \draw[dashed] (v) -- (w);
            \draw[dashed] (v) -- (c);
            \draw[dashed] (v) -- (d);
            
        \end{tikzpicture}
        \caption{\(\Gamma - V((M_2 \setminus \{cv\})\cup \{ab\})\)}
        \label{fig:lem:exclusive-subgraphs:4}
    \end{figure}
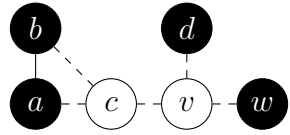

    \textbf{Case 1:} \(K_{1,3}\)

    Notice that \(a\) or \(b\) must be the degree \(3\) vertex.
    Suppose that the degree \(3\) vertex is \(b\).
    We then infer the edge connections in 
    Figure \ref{fig:lem:exclusive-subgraphs:5}.

    \begin{figure}[h!]
        \centering
        \begin{tikzpicture}
            \node [present] (a) at (0,0) {\(a\)};
            \node [present] (b) at (0,1) {\(b\)};
            \node [absent] (c) at (1,0) {\(c\)};
            \node [present] (d) at (2,1) {\(d\)};
            \node [absent] (v) at (2,0) {\(v\)};
            \node [present] (w) at (3,0) {\(w\)};
            \draw (a) -- (b);
            \draw[dashed] (b) -- (c);
            \draw[dashed] (c) -- (a);
            \draw[dashed] (v) -- (w);
            \draw[dashed] (v) -- (c);
            \draw[dashed] (v) -- (d);
            \draw (b) -- (d);
            \draw (b) -- (w);
        \end{tikzpicture}
        \caption{}
        \label{fig:lem:exclusive-subgraphs:5}
    \end{figure}

    Since \(d\) should be isolated in \(\Gamma - V(M)\), this is a contradiction.
    A similar contradiction arises if we use \(a\) as the degree \(3\) vertex instead
    since we would then have that \(d\) and \(a\) are adjacent.
    Hence, case 1 is impossible.

    \textbf{Case 2:} \(K_3 \sqcup K_1\)

    Since \(a\) and \(b\) are already adjacent,
    exactly one of \(d\) or \(w\) belongs to \(K_3\).
    If \(d\) belongs to \(K_3\), then we have to connect \(a\) to \(d\) contradicting
    that \(d\) is isolated in \(\Gamma - V(M)\).
    So, suppose \(w\) belongs to \(K_3\) (see Figure \ref{fig:lem:exclusive-subgraphs:6}).

    \begin{figure}[h!]
        \centering
        \begin{tikzpicture}
            \node [present] (a) at (0,0) {\(a\)};
            \node [present] (b) at (0,1) {\(b\)};
            \node [absent] (c) at (1,0) {\(c\)};
            \node [present] (d) at (2,1) {\(d\)};
            \node [absent] (v) at (2,0) {\(v\)};
            \node [present] (w) at (3,0) {\(w\)};
            \draw (a) -- (b);
            \draw[dashed] (b) -- (c);
            \draw[dashed] (c) -- (a);
            \draw[dashed] (v) -- (w);
            \draw[dashed] (v) -- (c);
            \draw[dashed] (v) -- (d);

            \draw (w) to [out=270, in=270] (a);

            \draw (w) -- (b);
        \end{tikzpicture}
        \caption{}
        \label{fig:lem:exclusive-subgraphs:6}
    \end{figure}

    Swapping \(cv\) for \(dv\), the resulting graph cannot be either \(K_{1,3}\) or \(K_3 \sqcup K_1\).
    Hence, case 2 is also impossible.
\end{proof}
\subsection{The Pseudomanifolds}
\label{subsection:pseudomanifolds}
Theorem \ref{thm:exclusive-subgraphs} gives \(\Gamma\) enough symmetry for us to be able to conclude what \(\Gamma\) must be. However, the arguments to do this are technical and do not rely on the assumption that \(\DConf_n(\Gamma)\) is an \(m\)-pseudomanifold. So, we state the relevant facts now and postpone their proofs until Section \ref{section:special-subgraphs}.

\begin{lem}
    \label{lem:K_r-is-special}
If \(\Gamma - V(M) \cong K_r\) for any \(m\)-matching \(M\) where \(r \ge 3\), then \(\Gamma \cong K_{2m+r}\).
\end{lem}
\begin{proof}
See subsection \ref{subsection:K_r-is-special}.
\end{proof}

\begin{lem}
\label{lem:K_p,q-is-special}
If \(\Gamma - V(M) \cong K_{p, q}\) for any \(m\)-matching \(M\)
where \(p \ge q \ge 1\) such that \(p + q \ge 4\), then \(\Gamma \cong K_{m+p, m+q}\).
\end{lem}
\begin{proof}
See subsection \ref{subsection:K_p,q-is-special}.
\end{proof}

\begin{lem}
\label{lem:K_3-cup-K_1-is-special}
If \(\Gamma - V(M) \cong K_3 \sqcup K_1\) for any \(m\)-matching \(M\),
then \(\Gamma \cong K_{p} \sqcup K_{q}\) for some odd \(p\) and \(q\) such that \(p + q = 2m + 4\).
\end{lem}
\begin{proof}
See subsection \ref{subsection:K_3-cup-K_1-is-special}.
\end{proof}

Since Theorem \ref{thm:exclusive-subgraphs} also tells us what \(n\) is, we now have the tools necessary to express exactly when \(\DConf_n(\Gamma)\) is an \(m\)-pseudomanifold, i.e.\ we can now prove Theorem \ref{thm:m-pseudomanifolds}.

\begin{proof}[Proof of Theorem \ref{thm:m-pseudomanifolds}]
If \(m = 1\) and \(n \in \{1, \abs{V(\Gamma)} - 1\}\), then Lemma
\ref{lem:easy-only-cycles} guarantees that
\(\Gamma\) must be the finite disjoint union of cycle graphs.
%That is, \(\Gamma \cong \bigsqcup_{k \in I} C_k\) for some finite multiset of integers \(\ge 3\).
Otherwise, Theorem \ref{thm:exclusive-subgraphs} guarantees that there exists some graph \(\Lambda \in \{K_3, K_{2,2}, K_{1,3}, K_{3}\sqcup K_1\}\) such that \(\Gamma - V(M) \cong \Lambda\) for every \((m-1)\)-matching \(M\) in \(\Gamma\).
For each possible case of \(\Lambda\), Theorem \ref{thm:exclusive-subgraphs} tells us what \(n\) can be, and either Lemma \ref{lem:K_r-is-special}, Lemma \ref{lem:K_p,q-is-special}, or Lemma \ref{lem:K_3-cup-K_1-is-special} tells us what \(\Gamma\) is.
\end{proof}

\subsection{The Manifolds}
\label{subsection:manifolds}
Theorem \ref{thm:m-pseudomanifolds} describes exactly when \(\DConf_n(\Gamma)\) is an \(m\)-pseudomanifold. We now investigate when \(\DConf_n(\Gamma)\)
satisfies the more stringent requirement of being an \(m\)-manifold.
By analyzing the link of certain cells in \(\DConf_n(\Gamma)\), Abrams showed the following.

\begin{prop}[Theorem 5.5 in \cite{abrams2000configurationspaces}]
  \label{prop:abrams-non-manifolds}
\(\DConf_m(K_{m+1, m+1})\) and \(\DConf_{m}(K_{2m+1})\) are not \(m\)-manifolds when \(m > 2\).
\end{prop}

Using the same approach, we find a similar problem happens with the other two families.
\begin{prop}
  \label{prop:remaining-non-manifolds2}
  \(\DConf_{m+1}(K_{2m + 1} \sqcup K_1)\) and
  \(\DConf_{m+1}(K_{m, m+2})\) are not \(m\)-manifolds
  when \(m > 2\).
\end{prop}

\begin{proof}
  We follow Abrams' approach by looking at the link of an \((m-3)\)-cell in \(\DConf_n(\Gamma)\).
  An \((m-3)\)-cell corresponds to a configuration of \(m-3\) particles moving simultaneously on \(\Gamma\).
  By examining the movements of the remaining \(n - (m-3)\) particles, 
  we can determine the Euler characteristic of the link of this cell.
  Specifically, the number of vertices, edges, and faces in the link correspond to configurations of \(n - (m-3)\)
  particles where \(1\), \(2\), and \(3\) particles are allowed to move simultaneously, respectively.
  We do not need to consider more than \(3\) particle movements since otherwise we would have
  \(m+1\) simultaneous movements contradicting Lemma \ref{lem:special-edges}.

  We first consider \((m-3)\)-cells of \(\DConf_{m+1}(K_{2m+1}\cup K_1)\). 
  Since if the link of any \((m-3)\)-cell fails to be a \(2\)-sphere, the space cannot be an \(m\)-manifold,
  we only consider the link of the \((m-3)\)-cell corresponding to \(n - (m - 3) = 4\) particles placed
  on \(K_{2m+1} \cup K_1 - V(M) = K_7\cup K_1\) (where \(M\) is some \((m-3)\)-matching).
  After placing \(4\) particles on the \(K_7\) component of \(K_7 \cup K_1\),
  there are \(12\), \(36\), and \(24\) ways for \(1\), \(2\), and \(3\) of these particles to move simultaneously, respectively.
  The discussion following Definition \ref{defn:matching-set} explains how to compute these numbers.
  Hence the Euler characteristic of the link is \(12 - 36 + 24 = 0\) meaning, the link is not a \(2\)-sphere.

  Next we consider \((m-3)\)-cells of \(\DConf_{m+1}(K_{m, m+2})\).
  Again we only look at the link of one kind of \((m-3)\)-cell, namely 
  \(n - (m - 3) = 4\) particles placed on the \(5\)-partition of \(K_{m, m+2} - V(M) = K_{3,5}\) (where \(M\) is some \((m-3)\)-matching).
  There are also \(12\), \(36\), and \(24\) ways for \(1\), \(2\), and \(3\) of these particles to move simultaneously, respectively.
  So, again the link is not a \(2\)-sphere.
\end{proof}

Lemma \ref{lem:special-edges} is not strong enough to determine whether \(\DConf_n(\Gamma)\) is a manifold. Specifically, Lemma \ref{lem:special-edges} requires that if we place \(m-1\) particles on \(\Gamma\)
such that each particle is able to independently move, there should be only two ways the remaining \(n - (m - 1)\) particles can move. Although \(K_p \sqcup K_q\) passes this test, if instead we first place down our \(n\) particles, then ask if the possible particle movements are conducive to \(\DConf_n(K_p\sqcup K_q)\) being a manifold, we find a problem if both \(p\) and \(q\) are greater than 1.

\begin{prop}
\label{prop:remaining-non-manifolds1}
Let \(p\) and \(q\) be odd positive integers such that \(p + q = 2m + 2\).
If both \(p\) and \(q\) are greater than \(1\), then
\(\DConf_{m+1}(K_p \sqcup K_q)\) is not an \(m\)-manifold.
\end{prop}

\begin{proof}
    Suppose \(p \ge q \ge 3\). Substituting \(p = 2m + 2 - q\) into \(q \le p\), we find \(q \le m+1\).
    If \(q = m+1\), then after placing \(m+1\) particles on the \(K_q\) component of \(K_p\sqcup K_q\), there is no way for any particle to move. Such a configuration corresponds to an isolated point in the configuration space which is not locally \(m\)-dimensional. 
    
    Otherwise, if \(q < m+1\), then we must have that \(p > m + 1\). Place \(q\) particles on the \(K_q\) component of \(K_p\sqcup K_q\). Then after placing the remaining \(n - q\) particles on the \(K_p\) component, notice that there are \(p - (n - q)\) vertices that the \(n - q\) particles on the \(K_p\) component can move to. Since \(n - q = (m+1) - q > 0\) and \(p - (n - q) = m+1\), there is at least one particle on the \(K_p\) component that can move to \(m+1\) different vertices. These \(m+1\) possible movements correspond to \(m+1\) different \(1\)-cells in the configuration space sharing a single \(0\)-cell which cannot happen in an \(m\)-manifold.
\end{proof}

The three previous propositions significantly reduce the potential number of \(m\)-manifolds. In fact, every remaining closed \(2\)-pseudomanifold was shown to be a closed \(2\)-manifold in \cite{abrams2000configurationspaces} and \cite{appiah2024algebraicstructurehyperbolicgraph}. 
We include a self-contained proof that each remaining space is a closed manifolds which proves the second main result of the paper.

\begin{proof}[Proof of Corollary \ref{cor:manifold-classification}]
  In light of Theorem \ref{thm:m-pseudomanifolds}, and Propositions
\ref{prop:abrams-non-manifolds}, \ref{prop:remaining-non-manifolds2}, and \ref{prop:remaining-non-manifolds1}, it remains to show that the remaining spaces are actually manifolds. 

  \textbf{Case 1:} 
  Since the one-point discretized configuration space is homeomorphic to the graph itself
  and a disjoint collection of cycles is a \(1\)-manifold, this case holds. 

  \textbf{Case 2:}
  In this case, either one or two particles are placed on the \(3\)-cycle.
  Since placing \(1\) or \(\abs{V(\Gamma)} - 1\) particles on a cycle graph \(\Gamma\)
  admits a 1-manifold for its discretized configuration space, this case holds.

  \textbf{Case 3:} 
  We simply draw the entire discretized configuration space in this case.
  See Figure \ref{fig:cor:manifold-classification:2:nbhd}. The labels on this graph correspond
  to ordered pairs of the vertices labeled in Figure \ref{fig:cor:manifold-classification:2:ygraph}.

  \begin{figure}[h!]
    \centering
    \begin{minipage}[b]{.4\textwidth}
      \centering
      \begin{tikzpicture}[
        mynode/.style={circle, draw},
      ]
        \node (1) at (0, 2) [mynode] {\(1\)};
        \node (2) at (2, 2) [mynode] {\(2\)};
        \node (3) at (1, 1) [mynode] {\(3\)};
        \node (4) at (1, -.3) [mynode] {\(4\)};

        \draw (1) -- (3);
        \draw (2) -- (3);
        \draw (3) -- (4);
      \end{tikzpicture}
      \caption{\(K_{1,3}\)}
      \label{fig:cor:manifold-classification:2:ygraph}
    \end{minipage}
    \begin{minipage}[b]{0.55\textwidth}
      \centering
      \begin{tikzpicture}[
        mynode/.style={circle, draw, fill=gray!20, font=\scriptsize, minimum size=26pt, inner sep=0pt },
      ]
        \node (12) at (0,  2) [mynode] {\(1, 2\)};
        \node (13) at (1.5,  2) [mynode] {\(1, 3\)};
        \node (14) at (3,  2) [mynode] {\(1, 4\)};
        \node (34) at (4.5,  2) [mynode] {\(3, 4\)};
        \node (24) at (6,  2) [mynode] {\(2, 4\)};
        \node (23) at (7.5, 2) [mynode] {\(2, 3\)};

        \node (21) at (7.5, 0) [mynode] {\(2, 1\)};
        \node (31) at (6, 0) [mynode] {\(3, 1\)};
        \node (41) at (4.5, 0) [mynode] {\(4, 1\)};
        \node (43) at (3, 0) [mynode] {\(4, 3\)};
        \node (42) at (1.5, 0) [mynode] {\(4, 2\)};
        \node (32) at (0, 0) [mynode] {\(3, 2\)};

        \draw (12) -- (13) -- (14) -- (34) -- (24) -- (23);
        \draw (21) -- (31) -- (41) -- (43) -- (42) -- (32);

        \draw (12) -- (32);
        \draw (23) -- (21);
      \end{tikzpicture}
      \caption{\(\DConf_2(K_{1,3})\)}
      \label{fig:cor:manifold-classification:2:nbhd}
    \end{minipage}
  \end{figure}

  \textbf{Case 4:}
  In Figure \ref{fig:cor:manifold-classification:3:K5} we label each vertex of \(K_5\).
  The shaded parts in Figure \ref{fig:cor:manifold-classification:3:nbhd} illustrate a neighborhood of \((1,2)\) in \(\DConf_2(K_5)\)
  whose interior locally resembles \(\mathbb{R}^2\). 
  We use the same convention as in case 2 where the labels in each vertex in the figure correspond to an ordered pair
  of vertices in the graph.
  Since any configuration in \(\DConf_2(K_5)\) belongs to the interior of a neighborhood like the one drawn,
  \(\DConf_2(K_5)\) is a surface without boundary.

  \begin{figure}[h!]
    \centering
    \begin{minipage}[b]{.45\textwidth}
      \centering
      \begin{tikzpicture}
        \foreach \i in {1,...,5} {
          \node (n\i) at ({(\i - 1)*72}:1.5) [circle, draw] {\(\i\)};
        }
        \foreach \i in {1,...,5} {
          \foreach \j in {\i,...,5} {
            \unless\ifnum\i=\j
              \draw (n\i) -- (n\j);
            \fi
          }
        }
      \end{tikzpicture}
      \vspace{.8cm}
      \caption{\(K_5\)}
      \label{fig:cor:manifold-classification:3:K5}
    \end{minipage}
    \begin{minipage}[b]{.45\textwidth}
      \centering
    
      \begin{tikzpicture}[
        mynode/.style={circle, draw, fill=gray!20, font=\scriptsize, minimum size=22pt, inner sep=0pt },
        grayfill/.style={gray!20,draw=black}
      ]
        \coordinate (c12) at (0, 0);
        \coordinate (c13) at (0:1.3);
        \coordinate (c43) at (30:2.3);
        \coordinate (c42) at (60:1.3);

        \filldraw[grayfill] (c12) -- (c13) -- (c43) -- (c42) -- cycle;

        \coordinate (c45) at (90:2.3);
        \coordinate (c15) at (120:1.3);

        \filldraw[grayfill] (c12) -- (c42) -- (c45) -- (c15) -- cycle;

        \coordinate (c35) at (150:2.3);
        \coordinate (c32) at (180:1.3);

        \filldraw[grayfill] (c12) -- (c32) -- (c35) -- (c15) -- cycle;

        \coordinate (c34) at (210:2.3);
        \coordinate (c14) at (240:1.3);

        \filldraw[grayfill] (c12) -- (c14) -- (c34) -- (c32) -- cycle;

        \coordinate (c54) at (270:2.3);
        \coordinate (c52) at (300:1.3);

        \filldraw[grayfill] (c12) -- (c14) -- (c54) -- (c52) -- cycle;

        \coordinate (c53) at (330:2.3);

        \filldraw[grayfill] (c12) -- (c52) -- (c53) -- (c13) -- cycle;

        \node (n12) at (c12) [mynode] {\(1,2\)};
        \node (n13) at (c13) [mynode] {\(1,3\)};
        \node (n43) at (c43) [mynode] {\(4,3\)};
        \node (n42) at (c42) [mynode] {\(4,2\)};

        \node (n45) at (c45) [mynode] {\(4,5\)};
        \node (n15) at (c15) [mynode] {\(1,5\)};

        \node (n35) at (c35) [mynode] {\(3,5\)};
        \node (n32) at (c32) [mynode] {\(3,2\)};

        \node (n34) at (c34) [mynode] {\(3,4\)};
        \node (n14) at (c14) [mynode] {\(1,4\)};

        \node (n54) at (c54) [mynode] {\(5,4\)};
        \node (n52) at (c52) [mynode] {\(5,2\)};

        \node (n53) at (c53) [mynode] {\(5,3\)};
      \end{tikzpicture}
      \caption{Neighborhood of \((1,2)\)}
      \label{fig:cor:manifold-classification:3:nbhd}
    \end{minipage}
  \end{figure}

  \textbf{Case 5:}
  Similar to Case 4, we label the vertices of \(K_{3,3}\) in Figure \ref{fig:cor:manifold-classification:4:K_3,3} and draw each 
  neighborhood type in Figure \ref{fig:cor:manifold-classification:4:nbhd}.

  \begin{figure}[h!]
    \centering
  \begin{tikzpicture}[
    mynode/.style={circle, draw}
  ]
  \node (n1) at (0,0) [mynode] {\(1\)};
  \node (n2) at (2,0) [mynode] {\(2\)};
  \node (n3) at (4,0) [mynode] {\(3\)};
  \node (n4) at (0,2) [mynode] {\(4\)};
  \node (n5) at (2,2) [mynode] {\(5\)};
  \node (n6) at (4,2) [mynode] {\(6\)};

  \draw (n1) -- (n4);
  \draw (n1) -- (n5);
  \draw (n1) -- (n6);

  \draw (n2) -- (n4);
  \draw (n2) -- (n5);
  \draw (n2) -- (n6);

  \draw (n3) -- (n4);
  \draw (n3) -- (n5);
  \draw (n3) -- (n6);
  \end{tikzpicture}

  \caption{\(K_{3,3}\)}
  \label{fig:cor:manifold-classification:4:K_3,3}
  \end{figure}

  \begin{figure}[h!]
    \centering
    \begin{minipage}[b]{.45\textwidth}
      \centering
      \begin{tikzpicture}[
        mynode/.style={circle, draw, fill=gray!20, font=\scriptsize, minimum size=22pt, inner sep=0pt },
        grayfill/.style={gray!20,draw=black}
      ]
        \coordinate (c12) at (0, 0);
    
        % --- Square 1 (Top-Right-ish) ---
        % 1 moves to 4 (at 0 deg), 2 moves to 5 (at 60 deg)
        \coordinate (c42) at (0:1.5);
        \coordinate (c45) at (30:2.5);
        \coordinate (c15) at (60:1.5);
    
        \filldraw[grayfill] (c12) -- (c42) -- (c45) -- (c15) -- cycle;
    
        % --- Square 2 (Top) ---
        % 2 moves to 5 (at 60 deg), 1 moves to 6 (at 120 deg)
        \coordinate (c65) at (90:2.5);
        \coordinate (c62) at (120:1.5);
    
        \filldraw[grayfill] (c12) -- (c15) -- (c65) -- (c62) -- cycle;
    
        % --- Square 3 (Top-Left-ish) ---
        % 1 moves to 6 (at 120 deg), 2 moves to 4 (at 180 deg)
        \coordinate (c64) at (150:2.5);
        \coordinate (c14) at (180:1.5);
    
        \filldraw[grayfill] (c12) -- (c62) -- (c64) -- (c14) -- cycle;
    
        % --- Square 4 (Bottom-Left-ish) ---
        % 2 moves to 4 (at 180 deg), 1 moves to 5 (at 240 deg)
        \coordinate (c54) at (210:2.5);
        \coordinate (c52) at (240:1.5);
    
        \filldraw[grayfill] (c12) -- (c14) -- (c54) -- (c52) -- cycle;
    
        % --- Square 5 (Bottom) ---
        % 1 moves to 5 (at 240 deg), 2 moves to 6 (at 300 deg)
        \coordinate (c56) at (270:2.5);
        \coordinate (c16) at (300:1.5);
    
        \filldraw[grayfill] (c12) -- (c52) -- (c56) -- (c16) -- cycle;
    
        % --- Square 6 (Bottom-Right-ish) ---
        % 2 moves to 6 (at 300 deg), 1 moves to 4 (at 0/360 deg)
        \coordinate (c46) at (330:2.5);
    
        \filldraw[grayfill] (c12) -- (c16) -- (c46) -- (c42) -- cycle;
    
        % --- Nodes ---
        \node (n12) at (c12) [mynode] {\(1,2\)};
    
        % Inner ring nodes
        \node (n42) at (c42) [mynode] {\(4,2\)};
        \node (n15) at (c15) [mynode] {\(1,5\)};
        \node (n62) at (c62) [mynode] {\(6,2\)};
        \node (n14) at (c14) [mynode] {\(1,4\)};
        \node (n52) at (c52) [mynode] {\(5,2\)};
        \node (n16) at (c16) [mynode] {\(1,6\)};
    
        % Outer ring nodes
        \node (n45) at (c45) [mynode] {\(4,5\)};
        \node (n65) at (c65) [mynode] {\(6,5\)};
        \node (n64) at (c64) [mynode] {\(6,4\)};
        \node (n54) at (c54) [mynode] {\(5,4\)};
        \node (n56) at (c56) [mynode] {\(5,6\)};
        \node (n46) at (c46) [mynode] {\(4,6\)};
      \end{tikzpicture}
    \end{minipage}
    \begin{minipage}[b]{.45\textwidth}
      \centering
      \begin{tikzpicture}[
        mynode/.style={circle, draw, fill=gray!20, font=\scriptsize, minimum size=22pt, inner sep=0pt },
        grayfill/.style={gray!20,draw=black}
      ]
        % Center Node (1,4)
        \coordinate (c14) at (0, 0);

        % --- Neighbors (Distance 1) ---
        % Right: Particle 2 moves 4 -> 3
        \coordinate (c13) at (1.5,0);
        % Top: Particle 1 moves 1 -> 5
        \coordinate (c54) at (0,1.5);
        % Left: Particle 2 moves 4 -> 2
        \coordinate (c12) at (-1.5,0);
        % Bottom: Particle 1 moves 1 -> 6
        \coordinate (c64) at (0,-1.5);

        % --- Corners (Distance 2 - The Squares) ---
        % Top-Right: P1 -> 5, P2 -> 3
        \coordinate (c53) at (1.5,1.5);
        % Top-Left: P1 -> 5, P2 -> 2
        \coordinate (c52) at (-1.5,1.5);
        % Bottom-Left: P1 -> 6, P2 -> 2
        \coordinate (c62) at (-1.5, -1.5);
        % Bottom-Right: P1 -> 6, P2 -> 3
        \coordinate (c63) at (1.5, -1.5);

        % --- Draw the 4 Squares ---
        % Square 1 (Top-Right)
        \filldraw[grayfill] (c14) -- (c13) -- (c53) -- (c54) -- cycle;
        % Square 2 (Top-Left)
        \filldraw[grayfill] (c14) -- (c54) -- (c52) -- (c12) -- cycle;
        % Square 3 (Bottom-Left)
        \filldraw[grayfill] (c14) -- (c12) -- (c62) -- (c64) -- cycle;
        % Square 4 (Bottom-Right)
        \filldraw[grayfill] (c14) -- (c64) -- (c63) -- (c13) -- cycle;

        % --- Place Nodes ---
        % Center
        \node (n14) at (c14) [mynode] {\(1,4\)};

        % Neighbors
        \node (n13) at (c13) [mynode] {\(1,3\)};
        \node (n54) at (c54) [mynode] {\(5,4\)};
        \node (n12) at (c12) [mynode] {\(1,2\)};
        \node (n64) at (c64) [mynode] {\(6,4\)};

        % Corners
        \node (n53) at (c53) [mynode] {\(5,3\)};
        \node (n52) at (c52) [mynode] {\(5,2\)};
        \node (n62) at (c62) [mynode] {\(6,2\)};
        \node (n63) at (c63) [mynode] {\(6,3\)};
      \end{tikzpicture}
      \vspace{.9cm}
    \end{minipage}
    \caption{Neighborhoods of \((1,2)\) and \((1,4)\) in \(\DConf_2(K_{3,3})\)}
    \label{fig:cor:manifold-classification:4:nbhd}
  \end{figure}

\textbf{Case 6:}
  The same argument as in previous cases applies here. See Figure \ref{fig:cor:manifold-classification:5:K_2,4}
  for how we label the vertices in \(K_{2,4}\) and Figures \ref{fig:cor:manifold-classification:5:nbhd1}
  and \ref{fig:cor:manifold-classification:5:nbhd2} for all neighborhood types (1, 0, and 2 particles placed in the 2-part set of \(K_{2,4}\).
  \begin{figure}[h!]
    \centering
    \begin{minipage}[b]{.45\textwidth}
      \centering
      \begin{tikzpicture}[
        mynode/.style={circle, draw}
      ]
        % Define the two main suspension vertices, positioned closer together
        \node (a1) at (0, 2) [mynode] {\(1\)};
        \node (a2) at (0, -2) [mynode] {\(6\)};

        % Define the intermediate vertices for the 'm' paths
        \node (b1) at (-2, 0) [mynode] {\(2\)};
        \node (bm) at (2, 0) [mynode] {\(5\)};

        % Draw the edges for the 'm' paths, passing through the b_i nodes
        \draw (a1) to[out=180, in=90] (b1);
        \draw (b1) to[out=270, in=180] (a2);

        \draw (a1) to[out=0, in=90] (bm);
        \draw (bm) to[out=270, in=0] (a2);

        \node (b2) at (-.667, 0) [mynode]{\(3\)};
        \node (b3) at (.667, 0) [mynode]{\(4\)};

        \draw (a1) to[out=225, in=90] (b2);
        \draw (a1) to[out=-45, in=90] (b3);

        \draw (a2) to[out=135, in=270] (b2);
        \draw (a2) to[out=45, in=270] (b3);
       \end{tikzpicture}
    \caption{\(K_{2,4}\) also known as \ \(\Theta_4\)}
    \label{fig:cor:manifold-classification:5:K_2,4}
  \end{minipage}
  \begin{minipage}[b]{.45\textwidth}
    \centering
    \begin{tikzpicture}[
      mynode/.style={circle, draw, fill=gray!20, font=\scriptsize, minimum size=26pt, inner sep=0pt },
      grayfill/.style={gray!20,draw=black}
    ]
      % Center Node (1,2,3)
      \coordinate (c123) at (0, 0);

      % Right: Particle 3 moves 3 -> 6
      \coordinate (c126) at (0:1.5);
      % Top: Particle 1 moves 1 -> 4
      \coordinate (c423) at (90:1.5);
      % Left: Particle 2 moves 2 -> 6
      \coordinate (c163) at (180:1.5);
      % Bottom: Particle 1 moves 1 -> 5
      \coordinate (c523) at (270:1.5);

      % use distance 2.121 = \sqrt{2} * 1.5
      % Top-Right: P1 -> 4, P3 -> 6
      \coordinate (c426) at (45:2.121);
      % Top-Left: P1 -> 4, P2 -> 6
      \coordinate (c463) at (135:2.121);
      % Bottom-Left: P1 -> 5, P2 -> 6
      \coordinate (c563) at (225:2.121);
      % Bottom-Right: P1 -> 5, P3 -> 6
      \coordinate (c526) at (315:2.121);

      % --- Draw the 4 Squares ---
      % Square 1 (Top-Right)
      \filldraw[grayfill] (c123) -- (c126) -- (c426) -- (c423) -- cycle;
      % Square 2 (Top-Left)
      \filldraw[grayfill] (c123) -- (c423) -- (c463) -- (c163) -- cycle;
      % Square 3 (Bottom-Left)
      \filldraw[grayfill] (c123) -- (c163) -- (c563) -- (c523) -- cycle;
      % Square 4 (Bottom-Right)
      \filldraw[grayfill] (c123) -- (c523) -- (c526) -- (c126) -- cycle;

      % --- Place Nodes ---
      % Center
      \node (n123) at (c123) [mynode] {\(1,2,3\)};

      % Neighbors
      \node (n126) at (c126) [mynode] {\(1,2,6\)};
      \node (n423) at (c423) [mynode] {\(4,2,3\)};
      \node (n163) at (c163) [mynode] {\(1,6,3\)};
      \node (n523) at (c523) [mynode] {\(5,2,3\)};

      % Corners
      \node (n426) at (c426) [mynode] {\(4,2,6\)};
      \node (n463) at (c463) [mynode] {\(4,6,3\)};
      \node (n563) at (c563) [mynode] {\(5,6,3\)};
      \node (n526) at (c526) [mynode] {\(5,2,6\)};
    \end{tikzpicture}
    \caption{Neighborhood of \((1,2,3)\)}
    \label{fig:cor:manifold-classification:5:nbhd1}
  \end{minipage}
\end{figure}

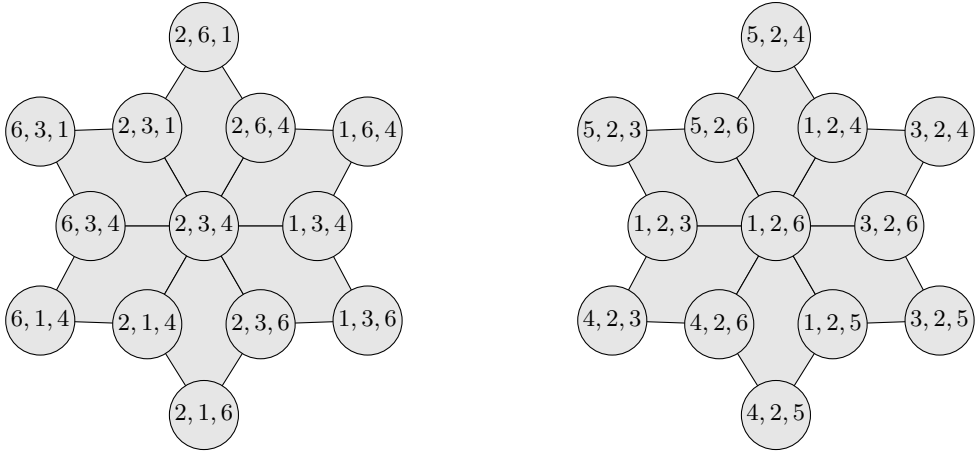
\begin{figure}[h!]
  \centering

  \begin{minipage}[b]{.45\textwidth}
    \begin{tikzpicture}[
      mynode/.style={circle, draw, fill=gray!20, font=\scriptsize, minimum size=26pt, inner sep=0pt },
      grayfill/.style={gray!20,draw=black}
    ]
      % Center Node (2,3,4)
      \coordinate (c234) at (0, 0);

      % --- Neighbors (Distance 1) ---
      % 0 deg: P1 -> 1
      \coordinate (c134) at (0:1.5);
      % 60 deg: P2 -> 6
      \coordinate (c264) at (60:1.5);
      % 120 deg: P3 -> 1
      \coordinate (c231) at (120:1.5);
      % 180 deg: P1 -> 6
      \coordinate (c634) at (180:1.5);
      % 240 deg: P2 -> 1
      \coordinate (c214) at (240:1.5);
      % 300 deg: P3 -> 6
      \coordinate (c236) at (300:1.5);

      % --- Corners (Distance 2 - The Squares) ---
      % 30 deg: P1->1, P2->6
      \coordinate (c164) at (30:2.5);
      % 90 deg: P2->6, P3->1
      \coordinate (c261) at (90:2.5);
      % 150 deg: P3->1, P1->6
      \coordinate (c631) at (150:2.5);
      % 210 deg: P1->6, P2->1
      \coordinate (c614) at (210:2.5);
      % 270 deg: P2->1, P3->6
      \coordinate (c216) at (270:2.5);
      % 330 deg: P3->6, P1->1
      \coordinate (c136) at (330:2.5);

      % --- Draw the 6 Squares ---
    
      % Square 1 (Right-ish)
      \filldraw[grayfill] (c234) -- (c134) -- (c164) -- (c264) -- cycle;
    
      % Square 2 (Top-Right-ish)
      \filldraw[grayfill] (c234) -- (c264) -- (c261) -- (c231) -- cycle;
    
      % Square 3 (Top-Left-ish)
      \filldraw[grayfill] (c234) -- (c231) -- (c631) -- (c634) -- cycle;
    
      % Square 4 (Left-ish)
      \filldraw[grayfill] (c234) -- (c634) -- (c614) -- (c214) -- cycle;
    
      % Square 5 (Bottom-Left-ish)
      \filldraw[grayfill] (c234) -- (c214) -- (c216) -- (c236) -- cycle;
    
      % Square 6 (Bottom-Right-ish)
      \filldraw[grayfill] (c234) -- (c236) -- (c136) -- (c134) -- cycle;

      % --- Place Nodes ---
      % Center
      \node (n234) at (c234) [mynode] {\(2,3,4\)};

      % Inner ring nodes (Neighbors)
      \node (n134) at (c134) [mynode] {\(1,3,4\)};
      \node (n264) at (c264) [mynode] {\(2,6,4\)};
      \node (n231) at (c231) [mynode] {\(2,3,1\)};
      \node (n634) at (c634) [mynode] {\(6,3,4\)};
      \node (n214) at (c214) [mynode] {\(2,1,4\)};
      \node (n236) at (c236) [mynode] {\(2,3,6\)};

      % Outer ring nodes (Corners)
      \node (n164) at (c164) [mynode] {\(1,6,4\)};
      \node (n261) at (c261) [mynode] {\(2,6,1\)};
      \node (n631) at (c631) [mynode] {\(6,3,1\)};
      \node (n614) at (c614) [mynode] {\(6,1,4\)};
      \node (n216) at (c216) [mynode] {\(2,1,6\)};
      \node (n136) at (c136) [mynode] {\(1,3,6\)};
    \end{tikzpicture}
  \end{minipage}
  \begin{minipage}[b]{.45\textwidth}
    \begin{tikzpicture}[
      mynode/.style={circle, draw, fill=gray!20, font=\scriptsize, minimum size=26pt, inner sep=0pt },
      grayfill/.style={gray!20,draw=black}
    ]
    % Center Node (1,2,6)
    \coordinate (c126) at (0, 0);

    % --- Neighbors (Distance 1) ---
    % 0 deg: P1 moves 1 -> 3
    \coordinate (c326) at (0:1.5);
    % 60 deg: P3 moves 6 -> 4
    \coordinate (c124) at (60:1.5);
    % 120 deg: P1 moves 1 -> 5
    \coordinate (c526) at (120:1.5);
    % 180 deg: P3 moves 6 -> 3
    \coordinate (c123) at (180:1.5);
    % 240 deg: P1 moves 1 -> 4
    \coordinate (c426) at (240:1.5);
    % 300 deg: P3 moves 6 -> 5
    \coordinate (c125) at (300:1.5);

    % --- Corners (Distance 2 - The Squares) ---
    % 30 deg: P1->3, P3->4
    \coordinate (c324) at (30:2.5);
    % 90 deg: P1->5, P3->4
    \coordinate (c524) at (90:2.5);
    % 150 deg: P1->5, P3->3
    \coordinate (c523) at (150:2.5);
    % 210 deg: P1->4, P3->3
    \coordinate (c423) at (210:2.5);
    % 270 deg: P1->4, P3->5
    \coordinate (c425) at (270:2.5);
    % 330 deg: P1->3, P3->5
    \coordinate (c325) at (330:2.5);

    % --- Draw the 6 Squares ---
    
    % Square 1 (P1->3, P3->4)
    \filldraw[grayfill] (c126) -- (c326) -- (c324) -- (c124) -- cycle;
    
    % Square 2 (P1->5, P3->4)
    \filldraw[grayfill] (c126) -- (c124) -- (c524) -- (c526) -- cycle;
    
    % Square 3 (P1->5, P3->3)
    \filldraw[grayfill] (c126) -- (c526) -- (c523) -- (c123) -- cycle;
    
    % Square 4 (P1->4, P3->3)
    \filldraw[grayfill] (c126) -- (c123) -- (c423) -- (c426) -- cycle;
    
    % Square 5 (P1->4, P3->5)
    \filldraw[grayfill] (c126) -- (c426) -- (c425) -- (c125) -- cycle;
    
    % Square 6 (P1->3, P3->5)
    \filldraw[grayfill] (c126) -- (c125) -- (c325) -- (c326) -- cycle;

    % --- Place Nodes ---
    % Center
    \node (n126) at (c126) [mynode] {\(1,2,6\)};

    % Inner ring nodes (Neighbors)
    \node (n326) at (c326) [mynode] {\(3,2,6\)};
    \node (n124) at (c124) [mynode] {\(1,2,4\)};
    \node (n526) at (c526) [mynode] {\(5,2,6\)};
    \node (n123) at (c123) [mynode] {\(1,2,3\)};
    \node (n426) at (c426) [mynode] {\(4,2,6\)};
    \node (n125) at (c125) [mynode] {\(1,2,5\)};

    % Outer ring nodes (Corners)
    \node (n324) at (c324) [mynode] {\(3,2,4\)};
    \node (n524) at (c524) [mynode] {\(5,2,4\)};
    \node (n523) at (c523) [mynode] {\(5,2,3\)};
    \node (n423) at (c423) [mynode] {\(4,2,3\)};
    \node (n425) at (c425) [mynode] {\(4,2,5\)};
    \node (n325) at (c325) [mynode] {\(3,2,5\)};

    \end{tikzpicture}
  \end{minipage}
  \caption{Neighborhoods of \((2,3,4)\) and \((1,2,6\))}
  \label{fig:cor:manifold-classification:5:nbhd2}
\end{figure}

\textbf{Case 7:}
  See Figure \ref{fig:cor:manifold-classification:6:graph}
  The neighborhood types in this case are \(2\) particles placed on the \(K_5\)-component and one particle
  placed on the \(K_1\)-component (exactly the same as Figure \ref{fig:cor:manifold-classification:3:nbhd})
  and \(3\) particles placed on the \(K_5\)-component (see Figure \ref{fig:cor:manifold-classification:6:nbhd}).

  \begin{figure}[h!]
    \centering
    \begin{minipage}[b]{.45\textwidth}
      \centering
      \begin{tikzpicture}
        \foreach \i in {1,...,5} {
          \node (n\i) at ({(\i - 1)*72}:1.5) [circle, draw] {\(\i\)};
        }
        \foreach \i in {1,...,5} {
          \foreach \j in {\i,...,5} {
            \unless\ifnum\i=\j
              \draw (n\i) -- (n\j);
            \fi
          }
        }
        \node (n6) at (36:2.5) [circle, draw] {\(6\)};
      \end{tikzpicture}
      \caption{\(K_5 \cup K_1\)}
      \label{fig:cor:manifold-classification:6:graph}
      \vspace{1cm}
    \end{minipage}
    \begin{minipage}[b]{.45\textwidth}
      \centering
      \begin{tikzpicture}[
        mynode/.style={circle, draw, fill=gray!20, font=\scriptsize, minimum size=26pt, inner sep=0pt },
        grayfill/.style={gray!20,draw=black}
      ]
        % Center Node (1,2,3)
        \coordinate (c123) at (0, 0);

        % --- Neighbors (Distance 1 - Single Moves) ---
        % 0 deg: P1 moves 1 -> 4
        \coordinate (c423) at (0:1.5);
        % 60 deg: P2 moves 2 -> 5
        \coordinate (c153) at (60:1.5);
        % 120 deg: P3 moves 3 -> 4
        \coordinate (c124) at (120:1.5);
        % 180 deg: P1 moves 1 -> 5
        \coordinate (c523) at (180:1.5);
        % 240 deg: P2 moves 2 -> 4
        \coordinate (c143) at (240:1.5);
        % 300 deg: P3 moves 3 -> 5
        \coordinate (c125) at (300:1.5);

        % --- Corners (Distance 2 - The Squares/Simultaneous Moves) ---
        % 30 deg: P1->4, P2->5
        \coordinate (c453) at (30:2.5);
        % 90 deg: P2->5, P3->4
        \coordinate (c154) at (90:2.5);
        % 150 deg: P3->4, P1->5
        \coordinate (c524) at (150:2.5);
        % 210 deg: P1->5, P2->4
        \coordinate (c543) at (210:2.5);
        % 270 deg: P2->4, P3->5
        \coordinate (c145) at (270:2.5);
        % 330 deg: P3->5, P1->4
        \coordinate (c425) at (330:2.5);

        % --- Draw the 6 Squares ---
    
        % Square 1 (P1->4, P2->5)
        \filldraw[grayfill] (c123) -- (c423) -- (c453) -- (c153) -- cycle;
    
        % Square 2 (P2->5, P3->4)
        \filldraw[grayfill] (c123) -- (c153) -- (c154) -- (c124) -- cycle;
    
        % Square 3 (P3->4, P1->5)
        \filldraw[grayfill] (c123) -- (c124) -- (c524) -- (c523) -- cycle;
    
        % Square 4 (P1->5, P2->4)
        \filldraw[grayfill] (c123) -- (c523) -- (c543) -- (c143) -- cycle;
    
        % Square 5 (P2->4, P3->5)
        \filldraw[grayfill] (c123) -- (c143) -- (c145) -- (c125) -- cycle;
    
        % Square 6 (P3->5, P1->4)
        \filldraw[grayfill] (c123) -- (c125) -- (c425) -- (c423) -- cycle;

        % --- Place Nodes ---
        % Center
        \node (n123) at (c123) [mynode] {\(1,2,3\)};

        % Inner ring nodes (Neighbors)
        \node (n423) at (c423) [mynode] {\(4,2,3\)};
        \node (n153) at (c153) [mynode] {\(1,5,3\)};
        \node (n124) at (c124) [mynode] {\(1,2,4\)};
        \node (n523) at (c523) [mynode] {\(5,2,3\)};
        \node (n143) at (c143) [mynode] {\(1,4,3\)};
        \node (n125) at (c125) [mynode] {\(1,2,5\)};

        % Outer ring nodes (Corners)
        \node (n453) at (c453) [mynode] {\(4,5,3\)};
        \node (n154) at (c154) [mynode] {\(1,5,4\)};
        \node (n524) at (c524) [mynode] {\(5,2,4\)};
        \node (n543) at (c543) [mynode] {\(5,4,3\)};
        \node (n145) at (c145) [mynode] {\(1,4,5\)};
        \node (n425) at (c425) [mynode] {\(4,2,5\)};
      \end{tikzpicture}
      \caption{Neighborhood of \((1,2,3)\)}
      \label{fig:cor:manifold-classification:6:nbhd}
    \end{minipage}
  \end{figure}

  Note that \(m = 1\) or \(m = 2\) in each case.
  Using the table in Figure \ref{fig:euler_characteristics} we can classify 
  each connected surface that appears in this theorem.
  Although the space is disconnected in case 7,
  notice that one component is homeomorphic to \(\DConf_3(K_5)\)
  and the other component is homeomorphic to \(\DConf_2(K_5)\).
\end{proof}
    \section{Some Special Subgraphs} \label{section:special-subgraphs}
In Theorem \ref{thm:m-pseudomanifolds}, we used the fact that we can often determine \(\Gamma\) if for any \((m-1)\)-matching we always know what \(\Gamma - V(M)\) is.
In this section, we provide proofs of these facts.
From this point forward in this section suppose \(m \ge 1\) and that \(\Gamma\) contains an \(m\)-matching.
Note that we are using \(m\)-matchings instead of \((m-1)\)-matchings as in section \ref{section:pseudomanifolds} to make our arguments more concise.

%
% ----- Input subsections -----
%

\subsection{The Complete Graph}
\label{subsection:K_r-is-special}

\begin{proof}[Proof of Lemma \ref{lem:K_r-is-special}]
    Let \(r \ge 3\) and suppose that \(\Gamma - V(M) \cong K_r\) for any \(m\)-matching \(M\). We will show that \(\Gamma \cong K_{2m + r}\).
    Notice that for any \(m\)-matching \(M\), \(\abs{V(\Gamma)} = \abs{V(M)} + \abs{V(K_r)} = 2m + r\).
    Now, let \(v_1\) and \(v_2\) be two distinct vertices in \(\Gamma\) and let \(M_0\) be some \(m\)-matching in \(\Gamma\).
    We will show that \(v_1\) and \(v_2\) are adjacent in \(\Gamma\). 

    If \(v_1\) and \(v_2\) belong to \(\Gamma - V(M_0)\), then they are adjacent since they belong to a complete graph.
    So, suppose without loss of generality that \(v_1\) belongs to \(V(M_0)\) and let \(w_1\) be the vertex adjacent to \(v_1\) in \(\Gamma(M_0)\).
    We proceed by cases on whether \(v_2\) belongs to \(V(M_0)\).

    \textbf{Case 1:} \(v_2\) does not belong to \(V(M_0)\).

    Since \(v_2\) does not belong to \(V(M_0)\) and \(r \ge 3\), there are at least two adjacent vertices \(u_1\) and \(u_2\) other than \(v_2\) in \(\Gamma - V(M_0)\).
    Let \(M_1 = (M_0 \cup \{u_1 u_2\}) \setminus \{v_1 w_1\}\).
    Since \(M_1\) is an \(m\)-matching, it follows that \(v_1\) and \(v_2\) are adjacent in the complete graph \(\Gamma - V(M_1)\).

    \textbf{Case 2:} \(v_2\) belongs to \(V(M_0)\).
    
    Let \(w_2\) be the vertex adjacent to \(v_2\) in \(\Gamma(M_0)\), and let
    \(u_1\) and \(u_2\) be two adjacent vertices in \(\Gamma - V(M_0)\).
    Define \(M_1 = (M_0 \cup \{u_1 u_2\}) \setminus \{v_2 w_2\}\).
    Then, \(M_1\) is an \(m\)-matching containing \(v_1\) but not \(v_2\) putting us in Case 1.
\end{proof}
\subsection{The Bipartite Graph}
\label{subsection:K_p,q-is-special}
\begin{proof}[Proof of Lemma \ref{lem:K_p,q-is-special}]
    Let \(p \ge q \ge 1\) such that \(p + q \ge 4\) and
    suppose that \(\Gamma - V(M) \cong K_{p,q}\) for any \(m\)-matching \(M\).
    We will show that \(\Gamma \cong K_{p + m, q + m}\).
    
    Notice that \(\abs{V(\Gamma)} = \abs{V(M)} + \abs{V(K_{p,q})} = 2m + p + q\).
    Also, \(p + p \ge p + q \ge 4\), meaning \(p \ge 2\).
    Let \(M_0\) be an \(m\)-matching, \(\{a_{m+i}\}_{i=1}^p\) be the vertices in the partite set of size \(p\) in \(\Gamma - V(M_0)\),
    and \(\{b_{m+j}\}_{j=1}^q\) be the vertices in the partite set of size \(q\) in \(\Gamma - V(M_0)\).
    %If \(p = q\), then ensure the vertices are labeled so that each \(a_{m+i}\) is adjacent to every \(b_{m+j}\). 

    Denote the \(m\) disjoint edges in \(M_0\) as \(v_1 w_1, v_2 w_2, \cdots, v_m w_m\).
    We define a new \(m\)-matching for each \(1 \le i \le m\) as follows.
    \[
         M_i = (M_0 \cup \{a_{m+1}b_{m+1}\}) \setminus \{ v_i w_i \}
    \]
    Since \(v_i\) and \(w_i\) are adjacent and \(\Gamma - V(M_i) \cong K_{p,q}\),
    one of these two vertices belongs to the partite set of size \(p\), and the other belongs to the partite set of size \(q\) in \(\Gamma - V(M_i)\).
    Let \(a_i\) be the vertex in the partite set of size \(p\) and \(b_i\) be the vertex in the partite set of size \(q\).
    % The critical property is the one that we force in the p = q case.
    Notice that if \(p \neq q\) then each \(a_i\) is adjacent to every \(b_{m+k}\) for \(2 \le k \le q\)
    and \(b_i\) is adjacent to every \(a_{m+k}\) for \(2 \le k \le p\). If \(p = q\), then choose each \(a_i\) and \(b_i\) so that these properties hold.
    Let \(A = \{a_i\}_{i=1}^{m+p}\) and \(B = \{b_j\}_{j=1}^{m+q}\).
    We claim that \(A\) and \(B\) form the partite sets of a complete bipartite graph. If this claim is true, then since \(A\) contains \(m+p\) vertices and \(B\) contains \(m + q\) vertices, it will follow that \(\Gamma \cong K_{p+m, q+m}\).
    Figure \ref{fig:lem:K_p,q-is-special:1} illustrates which vertices in \(\Gamma\) are known to be adjacent. The remaining argument consists of 3 steps.
    \begin{figure}[h!]
        \centering
        \begin{tikzpicture}[
              mynode/.style={circle, draw, minimum size=31pt, inner sep=0pt}
        ]

            \node [mynode] (a1) at (0,3) {\(a_1\)};
            \node (a1ai) at (1,3) {\(\dots\)};
            \node [mynode] (ai) at (2,3) {\(a_i\)};
            \node (aiaj) at (3,3) {\(\cdots\)};
            \node [mynode] (aj) at (4,3) {\(a_j\)};
            \node (ajamp1) at (5,3) {\(\cdots\)};
            \node [mynode] (amp1) at (6,3) {\(a_{m+1}\)};
            \node (amp1ampp) at (7,3) {\(\cdots\)};
            \node [mynode] (ampp) at (8,3) {\(a_{m+p}\)};

            \node [mynode] (b1) at (0,0) {\(b_1\)};
            \node (b1bj) at (1,0) {\(\dots\)};
            \node [mynode] (bi) at (2,0) {\(b_i\)};
            \node (bij) at (3,0) {\(\cdots\)};
            \node [mynode] (bj) at (4,0) {\(b_j\)};
            \node (bjbmq1) at (5,0) {\(\cdots\)};
            \node [mynode] (bmq1) at (6,0) {\(b_{m+1}\)};
            \node (bmq1bmq) at (7,0) {\(\cdots\)};

            \draw (b1) -- (a1);
            \draw (ai) -- (bi);
            \draw (aj) -- (bj);
            \draw (amp1) -- (bmq1);

            \draw (ampp) -- (b1);
            \draw (ampp) -- (bi);
            \draw (ampp) -- (bj);
            \draw (ampp) -- (bmq1);

        \end{tikzpicture}
        \caption{}
        \label{fig:lem:K_p,q-is-special:1}
    \end{figure}

    \textbf{Step 1:} No two vertices in \(A\) are adjacent in \(\Gamma\). 
    To see this, suppose for the sake of contradiction that \(a_i\) and \(a_j\) are adjacent in \(\Gamma\) for some \(i \neq j\).
    If both \(i \ge m + 1\) and \(j \ge m + 1\), then the vertices in \(\{a_i, a_j, b_{m+1}\}\) form a triangle in \(\Gamma - V(M_0)\) contradicting that \(\Gamma - V(M_0)\) is bipartite.
    So, assume \(i \le m\).
    If \(j > m + 1\), then the vertices in \(\{a_i, b_i, a_j\}\) form a triangle in \(\Gamma - V(M_i)\).
    So, also assume that \(j \le m + 1\).
    
    %If \(j = m+1\), then we break into two cases: \(q = 1\) and \(q > 1\). Assuming \(j = m+1\) and \(q = 1\), then let \(M_0' = (M_0 \setminus \{a_i b_i\}) \cup \{a_i a_{m+1}\}\). Notice \(b_i\) and \(b_{m+1}\) are adjacent to each \(a_{m+k}\) for \(2 \le k \le p\) in \(\Gamma - V(M_0')\), meaning that \(b_i\) and \(b_{m+1}\) belong to one partite set of \(\Gamma - V(M_0')\). However, \(q = 1\).
    %Assuming instead that \(j = m+1\) and \(q > 1\), notice the vertices in \(\{a_i, a_{m+1}, b_{m+2}\}\) form a triangle in \(\Gamma - V( (M_0 \setminus \{a_i b_i\}) \cup \{a_{m+2} b_{m+1}\})\).
    
    We define a new \(m\)-matching as follows.
    \[
        M'= (M_0 \cup \{a_{m+1} b_{m+1}, a_{m+p}b_i\}) \setminus \{a_i b_i, a_j b_j\}
    \]
    Notice that if \(a_i\) and \(b_j\) are adjacent in \(\Gamma - V(M')\), 
    then \(\{a_i, a_j, b_j\}\) forms a triangle in \(\Gamma - V(M')\), contradicting that \(\Gamma - V(M')\) is bipartite. So, \(a_i\) and \(b_j\) cannot be adjacent in \(\Gamma - V(M')\), meaning
    both vertices belong to the same partite set in \(\Gamma - V(M')\).
    We consider two cases.

    \textbf{Case 1:} \(q = 1\).
    Since \(p \ge 3\) in this case, the vertices \(a_{m+2}\) and \(a_{m+3}\) exist. Notice if \(j \le m\) both \(a_{m+2}\) and \(a_{m+3}\) are adjacent to \(b_j\) in \(\Gamma - V(M_j)\), and if \(j = m+1\), then \(a_{m+2}\) and \(a_{m+3}\) are adjacent to \(b_j\) in \(\Gamma - V(M_0)\). Since \(a_{m+2}\) and \(a_{m+3}\) also belong to \(\Gamma - V(M')\), both \(a_{m+2}\) and \(a_{m+3}\) are adjacent to \(b_j\) in \(\Gamma - V(M')\). So, \(a_{m+2}\) and \(a_{m+3}\) belong to the partite set in \(\Gamma - V(M')\) not including \(b_j\).
    However, now both \(a_i\) and \(b_j\) belong to one partite set in \(\Gamma - V(M')\), and both \(a_{m+2}\) and \(a_{m+3}\) belong to the other partite set of \(\Gamma - V(M')\), contradicting that \(q = 1\).

    \textbf{Case 2:} \(q \ge 2\).
    Note that \(b_{m+2}\) exists in this case.
    Observe that \(a_i\) is adjacent to \(b_{m+2}\) in \(\Gamma - V(M_i)\).
    Also, if \(j \le m\), then \(a_j\) is adjacent to \(b_{m+2}\) in \(\Gamma - V(M_j)\), and if \(j = m+1\), then \(a_j\) is adjacent to \(b_{m+2}\) in \(\Gamma - V(M_0)\). However \(a_i\), \(a_j\), and \(b_{m+2}\) belong to \(\Gamma - V(M')\). So, our assumption that \(a_i\) is adjacent to \(a_j\) implies that the vertices \(a_i\), \(a_j\), and \(b_{m+2}\) form a triangle in \(\Gamma - V(M')\) contradicting that \(\Gamma - V(M')\) is bipartite.

    \textbf{Step 2:} Every vertex in \(A\) is adjacent to every vertex in \(B\) in \(\Gamma\).
    To see this, let \(a_i \in A\) and \(b_j \in B\). 
    Note the following.
    \begin{enumerate}
        \item If \(i = j\), then \(a_i\) and \(b_j\) are adjacent since \(a_ib_j \in M_0\).
        \item If \(i > m + 1\), then \(a_i\) and \(b_j\) are adjacent in \(\Gamma - V(M_j)\).
        \item If \(j > m + 1\), then \(a_i\) and \(b_j\) are adjacent in \(\Gamma - V(M_i)\).
    \end{enumerate}
    So, suppose \(i\) and \(j\) satisfy none of the above conditions, i.e.\ suppose \(i \neq j\), \(i \le m +1\), and \(j \le m + 1\).
    We consider the same new \(m\)-matching as in step 1.
    \[
        M'= (M_0 \cup \{a_{m+1} b_{m+1}, a_{m+p}b_i\}) \setminus \{a_i b_i, a_j b_j\}
    \]
    Suppose for the sake of contradiction that \(a_i\) and \(b_j\) are not adjacent in \(\Gamma\). Then \(a_i\) and \(b_j\) are not adjacent in \(\Gamma - V(M')\), meaning that \(a_i\) and \(b_j\) belong to the same partite set in \(\Gamma - V(M')\).
    However, since \(a_j\) is adjacent to \(b_j\), the vertex \(a_j\) must belong to the other partite set in \(\Gamma - V(M')\) that does not include \(a_i\) and \(b_j\). Hence, \(a_i\) must be adjacent to \(a_j\), contradicting that \(A\) contains no adjacent vertices.

    \textbf{Step 3:} No two vertices in \(B\) are adjacent in \(\Gamma\).
    Suppose for the sake of contradiction that \(b_i\) and \(b_j\) are adjacent in \(\Gamma\) for some \(i \neq j\).
    Note the following.
    \begin{enumerate}
        \item If \(i, j > m +1\), then \(b_i\) and \(b_j\) are adjacent in \(\Gamma - V(M_0)\).
        \item If \(i > m + 1\) and \(j \le m+1\), then \(b_i\) and \(b_j\) are adjacent in \(\Gamma - V(M_j)\).
        \item If \(i \le m + 1\) and \(j > m + 1\), then \(b_i\) and \(b_j\) are adjacent in \(\Gamma - V(M_i)\).
    \end{enumerate}
    Based on how we originally chose which vertices belong to \(B\) before step 1, none of the above conditions can be true. That is, we must have \(i, j \le m+1\).
    We consider two cases.

    \textbf{Case 1:} \(q = 1\).
    We define a new \(m\)-matching.
    \[
        M^* = (M_0 \cup \{a_{m+1} b_{m+1}, b_i b_j\}) \setminus \{a_i b_i, a_j b_j\}
    \]
    Observe that there are \(m\) disjoint edges in \(M^*\) and that every vertex in \(B\)
    is incident to some edge in \(M^*\).
    Hence, \(V(\Gamma) \setminus V(M^*)\) only contains vertices in \(A\). So, \(\Gamma - V(M^*)\) cannot be \(K_{p,q}\) otherwise we would have adjacent vertices in \(A\).

    \textbf{Case 2:} \(q \ge 2\).
    We again define a new \(m\)-matching.
    \[
        M^{**} = (M_0 \cup \{a_{m+1} b_{m+1}, a_{m+2} b_{m+2}\}) \setminus \{a_i b_i, a_j b_j\}
    \]
    Notice \(\{a_i, b_i, b_j\}\) forms a triangle in \(\Gamma - V(M^{**})\) contradicting the fact that \(\Gamma - V(M^{**}) \cong K_{p,q}\) is bipartite.
    
\end{proof}

\subsection{The Complement of the Claw}
\label{subsection:K_3-cup-K_1-is-special}

\begin{proof}[Proof of Lemma \ref{lem:K_3-cup-K_1-is-special}]
    \tikzset{
        mynode/.style={circle, draw, minimum size=29pt, inner sep=1.5pt}
    };
    Suppose \(\Gamma - V(M) \cong K_3 \sqcup K_1\) for any \(m\)-matching \(M\). We will show that \(\Gamma \cong K_p \sqcup K_q\) for some odd integers \(p\) and \(q\) such that \(p + q = 2m + 4\).
    First, for any \(m\)-matching \(M\), notice that \(\abs{V(\Gamma - V(M))} = \abs{V(\Gamma)} - \abs{V(M)} = \abs{V(\Gamma)} - 2m = 4\).
    So, \(\abs{V(\Gamma)} = 2m + 4\).
    The remainder of the proof consists of 4 steps.
    
    \textbf{Step 1:} \(\Gamma\) is not connected.

    Suppose \(\Gamma\) is connected and let \(M_0\) be an \(m\)-matching in \(\Gamma\).
    Let \(\{a, b, c\}\) be the vertices in the \(K_3\) component of \(\Gamma - V(M_0)\)
    and let \(d\) be the isolated vertex in \(\Gamma - V(M_0)\).
    Since \(\Gamma\) is connected, there exists a path from \(a\) to \(d\) in \(\Gamma\).
    Let \(a = v_0, \cdots, v_k = d\) be the vertices in a path from \(a\) to \(d\).

    \begin{figure}[h!]
        \begin{tikzpicture}
        \begin{scope}[shift={(-1,0)}]
            \node [mynode] (a) at (0:1) {\(a\)};
            \node [mynode] (b) at (120:1) {\(b\)};
            \node [mynode] (c) at (240:1) {\(c\)};
            \draw (a) -- (b);
            \draw (b) -- (c);
            \draw (a) -- (c);
        \end{scope}
        \node[mynode] (v1) at (2,0) {\(v_1\)}; 
        \draw (a) -- (v1);
        \node[mynode] (v2) at (4,0) {\(v_2\)}; 
        \draw (v2) -- (v1) node [midway, above] {\(e\)};
        \node [mynode] (d) at (6,0) {\(d\)};
        \end{tikzpicture}
        \caption{}
        \label{fig:lem:K_3-cup-k_1-is-special:1}
    \end{figure}

    We can assume that \(v_1\) is not \(b\) or \(c\)
    otherwise relabel the \(K_3\) component of \(\Gamma - V(M_0)\) so that \(v_1\) is not \(b\) or \(c\).
    Notice that the path must have at least \(3\) vertices since \(d\) is not adjacent to \(a\) in \(\Gamma - V(M_0)\).
    Also, notice that \(v_1\) must belong to \(V(M_0)\) and exactly one edge \(e\) of \(M_0\) must 
    be incident to \(v_1\) in \(\Gamma\)
    (see Figure \ref{fig:lem:K_3-cup-k_1-is-special:1}).
    We consider two cases for \(e\).

    \textbf{Case 1:} \(e = v_1 v_2\).
    Define a new \(m\)-matching \(M_1\) by swapping the edge \(bc\) for \(v_1 v_2\).
    Since the vertices \(a = v_0\), \(v_1\), and \(v_2\) are connected in \(\Gamma - V(M_1)\), they must
    belong to the \(K_3\) component in \(\Gamma - V(M_1)\), meaning that \(a = v_0\) is adjacent to \(v_2\) (see Figure \ref{fig:lem:K_3-cup-k_1-is-special:2}).
    \begin{figure}[h!]
        \centering
        \begin{tikzpicture}
        \begin{scope}[shift={(-1,0)}]
            \node [mynode] (a) at (0:1) {\(a\)};
            \node [mynode] (b) at (120:1) {\(b\)};
            \node [mynode] (c) at (240:1) {\(c\)};
            \draw (a) -- (b);
            \draw (b) -- (c);
            \draw (a) -- (c);
        \end{scope}
        \node[mynode] (v1) at (2,0) {\(v_1\)}; 
        \draw (a) -- (v1);
        \node[mynode] (v2) at (4,0) {\(v_2\)}; 
        \draw (v2) -- (v1) node [midway, above] {\(e\)};
        \node [mynode] (d) at (6,0) {\(d\)};
        \draw (a) to [out=270, in=270] (v2);
        \end{tikzpicture}
        \caption{}
        \label{fig:lem:K_3-cup-k_1-is-special:2}
    \end{figure}

    \textbf{Case 2:} \(e = v_1 w\) for some \(w \neq v_2\).
    Similarly to case 1, define a new \(m\)-matching \(M_1\) by swapping the edge \(bc\) for \(v_1 w\).
    Since the vertices \(a = v_0\), \(v_1\), and \(w\) are connected in \(\Gamma - V(M_1)\), they must belong to the \(K_3\) component in \(\Gamma - V(M_1)\),
    meaning that \(a = v_0\) is adjacent to \(w\) (see Figure \ref{fig:lem:K_3-cup-k_1-is-special:3}).
    \begin{figure}[h!]
        \centering
        \begin{tikzpicture}
        \begin{scope}[shift={(-1,0)}]
            \node [mynode] (a) at (0:1) {\(a\)};
            \node [mynode] (b) at (120:1) {\(b\)};
            \node [mynode] (c) at (240:1) {\(c\)};
            \draw (a) -- (b);
            \draw (b) -- (c);
            \draw (a) -- (c);
        \end{scope}
        \node[mynode] (v1) at (2,0) {\(v_1\)}; 
        \draw (a) -- (v1);
        \node[mynode] (v2) at (4,0) {\(v_2\)}; 
        \draw (v1) -- (v2);
        \node [mynode] (d) at (6,0) {\(d\)};
        \node [mynode] (w) at (2, -2) {\(w\)};
        \draw (a) -- (w);
        \draw (v1) -- (w) node [midway, right] {\(e\)};
        \end{tikzpicture}
        \caption{}
        \label{fig:lem:K_3-cup-k_1-is-special:3}
    \end{figure}

    In either case, we have found a new \(m\)-matching \(M_1\) such that the path between some vertex in the \(K_3\) component of \(\Gamma - V(M_1)\) and \(d\) is shorter. Since the path is finite, repeatedly swapping edges like in one of the above two cases will eventually result in a matching \(M^*\) such that \(\Gamma - V(M^*)\) is connected, contradicting that \(d\) is isolated in \(\Gamma - V(M)\) for any \(m\)-matching \(M\).

    \textbf{Step 2:} Every component of \(\Gamma\) has an odd number of vertices.

    Let \(\Gamma_1, \Gamma_2, \ldots, \Gamma_k\) be the (non-empty) connected components of \(\Gamma\).
    Suppose for the sake of contradiction that \(\Gamma_1\) has an even number of vertices and
    let \(M_0\) be an \(m\)-matching in \(\Gamma\).
    Since \(\Gamma_1\) is connected and has at least 2 vertices, it must contain at least one edge \(e\).
    We proceed by cases on the number of vertices in \(\Gamma_1 - V(M_0)\).

    \textbf{Case 1:} \(\Gamma_1 - V(M_0)\) contains \(0\) vertices.
    In this case \(e\) must belong to \(M_0\).
    Let \(M_1\) be an \(m\)-matching created by swapping the edge \(e\) for some edge in \(\Gamma - V(M_0)\).
    Observe that \(\Gamma - V(M_1)\) contains the connected component \(\Gamma_1 - V(M_1)\) which consists of exactly 2 vertices. Since \(K_3\cup K_1\) contains no such components, we have reached a contradiction.

    \textbf{Case 2:} \(\Gamma_1 - V(M_0)\) contains at least \(1\) vertex.
    Since \(\Gamma_1\) has an even number of vertices and \(M_0\) contains disjoint edges,
    \(\Gamma_1 - V(M_0)\) must have an even number of vertices.
    However, since no component of \(K_3 \cup K_1\) has an even number of vertices, we reach a contradiction.

    \textbf{Step 3:} There are exactly two components in \(\Gamma\).

    Let \(\Gamma_1, \Gamma_2, \ldots, \Gamma_k\) be the (non-empty) connected components of \(\Gamma\).
    Step 1 guarantees that \(k \ge 2\), so suppose that \(k \ge 3\).
    As shown in step 2, each component has an odd number of vertices. Since \(V(M)\) contains an even number of vertices, \(\Gamma - V(M)\) has at least \(3\) connected components
    each with an odd number of vertices
    for any \(m\)-matching \(M\). 
    However, since \(K_3 \cup K_1\) has
    exactly \(2\) connected components, we reach a contradiction.

    \textbf{Step 4:} Each component of \(\Gamma\) is complete.

    Let \(\Gamma_1\) and \(\Gamma_2\) be the two connected components of \(\Gamma\)
    and let \(p = \abs{V(\Gamma_1)}\) and \(q = \abs{V(\Gamma_2)}\).
    Without loss of generality, suppose that \(p \ge q\).
    We consider two cases for the value of \(q\).

    \textbf{Case 1:} \(q = 1\).
    In this case \(\Gamma_2 \cong K_1\).
    So, any \(m\)-matching \(M\) in \(\Gamma\) must consist of edges taken solely from \(\Gamma_1\).
    Hence \(\Gamma - \Gamma_2 - V(M) \cong K_3\) for any \(m\)-matching \(M\).
    Lemma \ref{lem:K_r-is-special} then guarantees that \(\Gamma - \Gamma_2\) is complete.
    Since \(\Gamma_1\) and \(\Gamma_2\) are disjoint, \(\Gamma_1 - \Gamma_2 = \Gamma_1\). So, \(\Gamma_1\) is complete i.e.\ \(\Gamma \cong K_p \sqcup K_1\).

    \textbf{Case 2:} \(q \ge 3\).

    Let \(M_0\) be an \(m\)-matching in \(\Gamma\).
    In the same fashion as Step 1, we find that the component of \(\Gamma\) containing the \(K_3\) component of \(\Gamma - V(M_0)\) cannot also contain the \(K_1\) component of \(\Gamma - V(M_0)\).
    Let \(\Gamma_i\) be the connected component of \(\Gamma\)
    containing the \(K_3\) component of \(\Gamma - V(M_0)\). 
    We claim that the \(K_1\) component of \(\Gamma - V(M_0)\) cannot belong to \(\Gamma_i\). To see this, let \(\{a,b,c\}\) be the vertices of the \(K_3\) component of \(\Gamma - V(M_0)\) and \(d\) be the isolated vertex of the \(K_1\) component of \(\Gamma - V(M_0)\).
    Since \(\Gamma_i\) is connected, there exists a path between \(a\) and \(d\).
    The same sequence of \(m\)-matchings described in step 1 results in an \(m\)-matching \(M^*\) such that \(\Gamma - V(M^*)\) is connected, contradicting that we should have \(\Gamma - V(M^*) \cong K_3 \sqcup K_1\).
    Hence, the \(K_3\) component of \(\Gamma - V(M_0)\) belongs to \(\Gamma_i\), and the \(K_1\) component of \(\Gamma - V(M_0)\) must belong to the other component \(\Gamma_j\) of \(\Gamma\).

    Let \(v_1\) and \(v_2\) be two distinct vertices in \(\Gamma_i\). We will show that \(v_1\) is necessarily adjacent to \(v_2\).
    If \(v_1\) and \(v_2\) both belong to the \(K_3\) component in \(\Gamma - V(M_0)\), then immediately \(v_1\) and \(v_2\) are adjacent. So, suppose \(v_1\) does not belong to this \(K_3\) component, meaning that \(v_1\) must belong to \(V(M_0)\). Let \(w_1\) be the vertex such that \( v_1 w_1 \in M_0\), and
    let \(u_1 \neq v_2\) and \(u_2 \neq v_2\) be two vertices in the \(K_3\) component of \(\Gamma - V(M_0)\).

    \begin{figure}[h!]
        \centering
        \begin{tikzpicture}
            \begin{scope}[shift={(-5,0)}]
                \node [mynode] (v1) at (120:1) {\(v_1\)};
                \node [mynode] (w1) at (240:1) {\(w_1\)};
                \draw (v1) -- (w1);
            \end{scope}
            \begin{scope}[shift={(-3,0)}]
                \node [mynode] (v2) at (120:1) {\(v_2\)};
                \node [mynode] (w2) at (240:1) {\(w_2\)};
                \draw (v2) -- (w2);
            \end{scope}
            \begin{scope}[shift={(-1,0)}]
                \node [mynode] (u1) at (120:1) {\(u_1\)};
                \node [mynode] (u2) at (240:1) {\(u_2\)};
                \node [mynode] (u3) at (0:1) {};
                \draw (u1) -- (u2);
                \draw (u2) -- (u3);
                \draw (u1) -- (u3);
            \end{scope}
        \end{tikzpicture}
        \caption{\(\Gamma_i - V(M_0)\), \(v_1 w_1\), and \(v_2 w_2\) when \(v_2 \in V(M_0)\)}
        \label{fig:lem:K_3-cup-k_1-is-special:4}
    \end{figure}
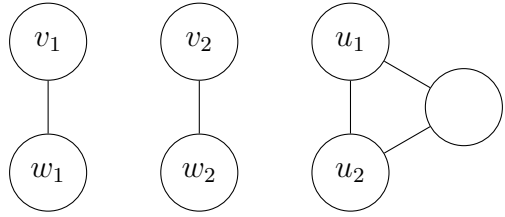
    
    We claim that there is a \(m\)-matching \(M_1\) in \(\Gamma\) where \(v_2\) belongs to the \(K_3\) component of \(\Gamma - V(M_1)\).
    To see this, note that if \(v_2\) already belongs to the \(K_3\) component of \(\Gamma - V(M_0)\), then let \(M_1 = M_0\). Otherwise, assuming
    that \(v_2\) is not in the \(K_3\) component of \(\Gamma - V(M_0)\) (see Figure \ref{fig:lem:K_3-cup-k_1-is-special:4}), since \(v_2\) is a vertex in \(\Gamma_i\), the vertex \(v_2\) must belong to \(V(M_0)\). Let \(w_2\) be the vertex where \(v_2 w_2 \in M_0\). If \(w_2 = v_1\) then \(v_1\) is adjacent to \(v_2\) and we are done. So, assume \(w_2 \neq v_1\).
    Let \(M_1\) be the \(m\)-matching formed by swapping the edge \(v_2 w_2\) for the edge \(u_1 u_2\).

    Now, let \(t_1\) and \(t_2\) be the two vertices in the \(K_3\) component of \(\Gamma - V(M_1)\) distinct from \(v_2\) (see Figure \ref{fig:lem:K_3-cup-k_1-is-special:5}).
    Let \(M_3\) be the \(m\)-matching formed by swapping the edge \(v_1 w_1\) for the edge \(t_1 t_2\). Then, \(\{v_1, w_1, v_2\}\) forms the \(K_3\) component of \(\Gamma - V(M_3)\). Hence, \(v_1\) is adjacent to \(v_2\).

    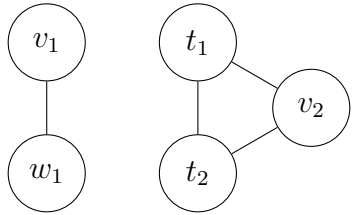
\begin{figure}[h!]
        \centering
        \begin{tikzpicture}
            \begin{scope}[shift={(-3,0)}]
                \node [mynode] (v1) at (120:1) {\(v_1\)};
                \node [mynode] (w1) at (240:1) {\(w_1\)};
                \draw (v1) -- (w1);
            \end{scope}

            \begin{scope}[shift={(-1,0)}]
                \node [mynode] (t1) at (120:1) {\(t_1\)};
                \node [mynode] (t2) at (240:1) {\(t_2\)};
                \node [mynode] (v2) at (0:1) {\(v_2\)};
                \draw (t1) -- (t2);
                \draw (t1) -- (v2);
                \draw (t2) -- (v2);
            \end{scope}
        \end{tikzpicture}
        \caption{\(\Gamma_i - V(M_1)\)}
        \label{fig:lem:K_3-cup-k_1-is-special:5}
    \end{figure}
    
    Finally, we show that \(\Gamma_j\) is complete. Since \(\Gamma_j\) is connected and \(p \ge q \ge 3\), \(\Gamma_j\) contains at least one edge \(e\).
    Since \(\Gamma_j\) contains the isolated vertex of \(\Gamma - V(M_0)\), one of the edges in \(\Gamma_j\) must belong to \(M_0\).
    Let \(M_4\) be an \(m\)-matching obtained by swapping \(e\) for some edge in the \(K_3\) component of \(\Gamma - V(M_0)\).
    Notice that the \(K_3\) component of \(\Gamma - V(M_2)\) belongs to \(\Gamma_j\) and the \(K_1\) component of \(\Gamma - V(M_2)\) belongs to \(\Gamma_i\).
    Replacing \(\Gamma_i\) with \(\Gamma_j\) and \(M_0\) with \(M_4\) in the previous three paragraphs shows that \(\Gamma_j\) must also be complete. Hence \(\Gamma \cong K_p \sqcup K_q\).
\end{proof}

    \section{Questions} 
\label{section:questions}
We end this paper with some open questions and brief investigations of them. 

%
% ----- Input subsections -----
%

\subsection{Other Pseudomanifolds?}
Theorem \ref{thm:m-pseudomanifolds} only classifies the closed \(m\)-pseudomanifolds without boundary.
\begin{question}
    If \(\Gamma\) is a simple graph without loops, when is \(\DConf_n(\Gamma)\) a closed \(m\)-pseudomanifold \textit{with} boundary?
\end{question}
One approach to this question is to try and extend the technique of Lemma \ref{lem:special-edges}
to account for boundary points.
A boundary point would correspond to a configuration in \(\Gamma - V(M)\) where only one particle can move. This gives a third possibility for the edges guaranteed by this lemma. Alternatively, one may be able to think about the known pseudomanifolds and see what happens when boundary is introduced by cutting out holes in the space.

Molly Ison's work in \cite{ison2005two} classifies when \(\DConf_2(\Gamma)\)
is a 2-pseudomanifold for both connected simple and connected non-simple graphs without loops.
We collect her results in Theorem \ref{thm:ison-pseudomanifolds} below.

\begin{thm}
    \label{thm:ison-pseudomanifolds}
    \cite{ison2005two}
    Let \(\Gamma\) be a connected graph without loops (potentially with multiple edges). Then \(\DConf_2(\Gamma)\) is a 2-pseudomanifold with or without boundary if and only if \(\Gamma\) is one of the following.
    \begin{enumerate}
        \item The cycle graph on \(n\) vertices when \(n \ge 4\).
        \item One of the two graphs shown in Figure \ref{fig:thm:ison-pseudomanifolds}. 
        \item A simple graph on \(5\) vertices where each vertex has degree at least \(2\)
        and no \(3\)-cycle contains a vertex with degree exactly \(2\).
        \item \(K_4\)
        \item \(K_{2,2}\) or \(K_4\) with each edge doubled.
    \end{enumerate}
\end{thm}

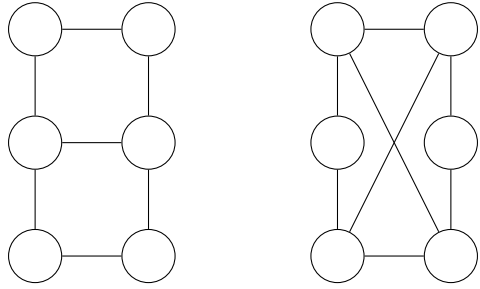
\begin{figure}[h!]
  \centering
  \begin{tikzpicture}
    \begin{scope}
      \node (v1) at (0,0) [dormant_graph_node] {};
      \node (v2) at (0,1.5) [dormant_graph_node] {};
      \node (v3) at (0,3) [dormant_graph_node] {};
      \node (v4) at (1.5,0) [dormant_graph_node] {};
      \node (v5) at (1.5,1.5) [dormant_graph_node] {};
      \node (v6) at (1.5,3) [dormant_graph_node] {};
      \draw (v1) to (v4);
      \draw (v1) to (v2);
      \draw (v2) to (v5);
      \draw (v2) to (v3);
      \draw (v3) to (v6);
      \draw (v4) to (v5);
      \draw (v5) to (v6);
    \end{scope}
    \begin{scope}[xshift=4cm]
      \node (v1) at (0,0) [dormant_graph_node] {};
      \node (v2) at (0,1.5) [dormant_graph_node] {};
      \node (v3) at (0,3) [dormant_graph_node] {};
      \node (v4) at (1.5,0) [dormant_graph_node] {};
      \node (v5) at (1.5,1.5) [dormant_graph_node] {};
      \node (v6) at (1.5,3) [dormant_graph_node] {};
      \draw (v1) to (v2);
      \draw (v1) to (v4);
      \draw (v2) to (v3);
      \draw (v3) to (v6);
      \draw (v4) to (v5);
      \draw (v5) to (v6);
      \draw (v1) to (v6);
      \draw (v3) to (v4);
    \end{scope}
  \end{tikzpicture}
  \caption{Two graphs where \(\DConf_2(\Gamma)\) is a 2-pseudomanifold with boundary}
  \label{fig:thm:ison-pseudomanifolds}
\end{figure}

Theorem \ref{thm:m-pseudomanifolds} only classifies the spaces admitted by \textit{simple} graphs.

\begin{question}
If \(\Gamma\) is a graph containing loops or multiple edges, when is \(\DConf_n(\Gamma)\) a closed \(m\)-pseudomanifold?
\end{question}

Extending the possibilities of Lemma \ref{lem:special-edges} is immediately fruitful
in the case of multiple edges.
Suppose that there are two vertices \(v\) and \(w\) in \(\Gamma - V(M)\) with multiple edges
between them such that
one of the \(n - (m - 1)\) particles placed on \(\Gamma - V(M)\) is placed on \(v\)
and is free to move to \(w\). 
Since there must be exactly two particle movements available
and at least two are already available from \(v\) to \(w\),
there must be exactly two edges between \(v\) and \(w\).
Furthermore, once this particle moves to \(w\), there cannot be any possible way for this particle to move other than along one of the two edges back to \(v\).

\begin{figure}[h!]
  \centering
    \begin{tikzpicture}
        \node (v) at (0,0) [circle, draw] {\(v\)};
        \node (w) at (2,0) [circle, draw] {\(w\)};
        \draw (v) to [out=45, in=135] (w);
        \draw (v) to [out=-45, in=-135] (w);
    \end{tikzpicture}
    \caption{\(\Gamma - V(M)\)}
    \label{fig:multiple-edges}
\end{figure}
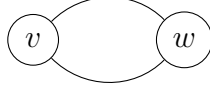

Hence one component of \(\Gamma - V(M)\) is just the vertices \(v\) and \(w\)
with two edges between them (see Figure \ref{fig:multiple-edges}).
Furthermore, there cannot be any isolated vertices in \(\Gamma - V(M)\). To see this, suppose that there was an isolated vertex \(u\). If there is a particle placed at \(u\), then that particle could be moved to one of \(v\) or \(w\) leaving no way for any particle to move.
If there is no particle placed at \(u\), then the particle at \(v\) or \(w\) could be moved to \(u\), again leaving no way for any particle to move.

One can check that \(K_4\) and \(K_{2,2}\) with each edge doubled satisfy the condition that \(\Gamma - V(M)\) is just the graph in Figure \ref{fig:multiple-edges} when \(m = n = 2\).
These examples agree with Theorem \ref{thm:ison-pseudomanifolds}.
However, it is unknown exactly which non-simple graphs work for other
values of \(m\) and \(n\).

\subsection{Other Special Subgraphs?}
The graphs in Section \ref{section:special-subgraphs} have lots of symmetry.
If we do not require that \(\DConf_n(\Gamma)\) is an \(m\)-pseudomanifold,
then perhaps there are other graphs that we can reconstruct solely from knowing what the deleted vertex matching subgraphs are. We formalize the idea that a subgraph is special in this way in Definition \ref{defn:residual-sequence} below.

\begin{defn}
\label{defn:residual-sequence}
Given a graph \(\Lambda\),
we say a sequence of graphs \(\{\Gamma_m\}_{m=1}^{\infty}\) is \(\Lambda\)-residual if and only if for each positive integer \(m\) the following are true.
\begin{enumerate}
\item \(\Gamma_m\) contains at least one \(m\)-matching.
\item \(\Gamma_m - V(M) \cong \Lambda\) for every \(m\)-matching \(M\) in \(\Gamma_m\).
\end{enumerate}
If a \(\Lambda\)-residual sequence exists, we say \(\Lambda\) is a special subgraph.
\end{defn}

\begin{question}
    \label{qst:special-graphs}
    Which graphs are special subgraphs?
\end{question}

A trivial example of a \(\Lambda\)-residual sequence is \(\{mK_2\sqcup \Lambda\}_{m=1}^{\infty}\) when \(\Lambda\) has no edges. Translating the results from Section \ref{section:special-subgraphs}, Lemmas \ref{lem:K_r-is-special} and \ref{lem:K_p,q-is-special} guarantee that if \(\Lambda\) is \(K_r\) when \(r \ge 3\), or \(K_{p,q}\) when \(p + q \ge 4\), then only one \(\Lambda\)-residual sequence exists: \(\{K_{2m+r}\}_{m=1}^\infty\) or \(\{K_{p+m, q+m}\}_{m=1}^{\infty}\), respectively.
In the case when \(\Lambda = K_3\sqcup K_1\), Lemma \ref{lem:K_3-cup-K_1-is-special} guarantees each term in a \(\Lambda\)-residual sequence has the form of \(K_p\sqcup K_q\) where \(p\) and \(q\) are odd, and \(p + q = 2m + 4\).

The process of finding an \(m\)-matching \(M\) in \(\Gamma\) and constraining the structure of \(\Gamma - V(M)\) parallels other developments in matching theory.
For example, much of the literature builds on the idea of \(m\)-extendability. A connected graph \(\Gamma\) is \(m\)-extendable \cite{plummerNextendableGraphs1980} if it contains at least one \(m\)-matching, and for any \(m\)-matching \(M\) in \(\Gamma\), the subgraph \(\Gamma - V(M)\) contains a perfect matching. So, if \(\Lambda\) contains a perfect matching and \(\Gamma_m\) is connected, then \(\Gamma_m\) is \(m\)-extendable. See Plummer's survey \cite{plummerRecentProgressMatching2008} for important developments. However, many processes operate in an order opposite to ours where first \(\Gamma\) is constrained, then some matching is found in \(\Gamma\). 
For example, \(\Gamma\) is said to be \(\Lambda\)-removable if for any subgraph \(\Lambda' \cong \Lambda\) in \(\Gamma\), the subgraph \(\Gamma - V(\Lambda')\) contains a perfect matching.
In \cite{chenRemovableForcedSubgraphs2024} the authors look at a special case of removability and determine that a graph is \(mK_2\)-forced if and only if it is \(K_{2m +2}\) or \(K_{m+1,m+1}\). Generally, \(\Gamma\) is a \(\Lambda\)-forced graph if it is \(\Lambda\)-removable and the perfect matching in each vertex-deleted subgraph is unique. Unfortunately, being \(\Lambda\)-removable or \(\Lambda\)-forced is generally not sufficient or necessary for being a term in a \(\Lambda\)-residual sequence. However, the sequences \(\{K_{2m+2}\}_{m=1}^{\infty}\) and \(\{K_{m+1,m+1}\}_{m=1}^{\infty}\) both form \(K_2\)-residual sequences.

Uniqueness is a natural question. We already know that there are unique \(K_r\)-residual and \(K_{p,q}\)-residual sequences when \(r \ge 3\) and \(p + q \ge 4\).
An easy nonexample other than when \(\Lambda = K_3\sqcup K_1\) is if \(\Lambda\) has at least one vertex and no edges. In this case, \(\Lambda \cong nK_1\) for some \(n \ge 1\), and \(\{K_3 \sqcup (m-1)K_2 \sqcup (n-1)K_1\}_{m=1}^{\infty}\) is \(\Lambda\)-residual. Since \(\{mK_2\sqcup nK_1\}_{m=1}^{\infty}\) is also \(nK_1\)-residual, \(nK_1\) does not have a unique residual sequence when \(n \ge 1\).

\begin{question}
For which graphs \(\Lambda\) is there a unique \(\Lambda\)-residual sequence?
\end{question}

In the proof of Lemma \ref{lem:manifold-with-only-cycles}, we saw that sufficiently large cycle graphs are not special subgraphs. We extract this idea into Proposition \ref{prop:cycles-not-special} below.

\begin{prop}
    \label{prop:cycles-not-special}
    If \(r \ge 5\), then \(C_r\) is not a special subgraph.
\end{prop}

\begin{proof}
    Suppose \(\{\Gamma_m\}_{m=1}^{\infty}\) is \(C_5\)-residual.
    In the same fashion as in Lemma \ref{lem:manifold-with-only-cycles}, Figure \ref{fig:prop:cycles-not-special:1} illustrates
    a sequence (top left, top right, bottom left, then bottom right) of ``swaps and inferences'' performed on \(\Gamma_1\) that contradict that \(\Gamma_1 - V(M)\) is \(C_5\) for any \(1\)-matching \(M\). The same argument works for larger values of \(r\).

    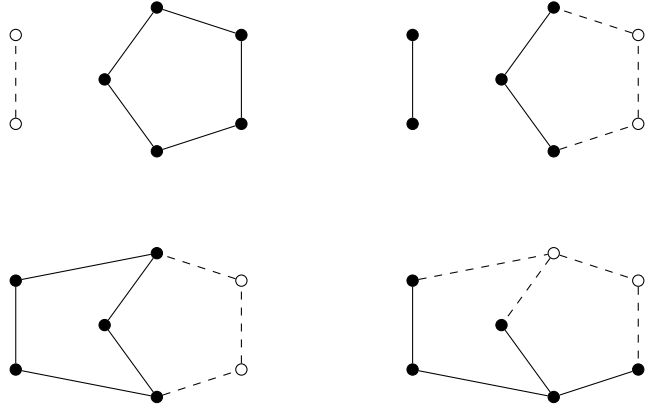
\begin{figure}[h!]
        \centering
        \begin{tikzpicture}
            % active and dormant node styles for vertices present/ not present after deleting the matching
            \tikzset{
                active/.style = {circle, draw=black, fill=black, inner sep = 1.5pt, text=white},
                dormant/.style = {circle, draw=black, fill=white, inner sep = 1.5pt, text=black}
            }
        
            % command for drawing the 7 vertices used in the argument.
            % Usage: \DrawVertices{<active nodes>}{<dormant nodes>}
            \newcommand{\DrawVertices}[2]{
                % extra edge
                \coordinate (c1) at (-2.176, -.588);
                \coordinate (c2) at (-2.176, .588);
                % cycle
                \foreach \i in {3,4,5,6,7} {
                    \coordinate (c\i) at ({(\i - 3) * 72 - 36}:1);
                }
        
                \foreach \i in {#1} {
                    \node (n\i) at (c\i) [active] {};
                }
                \foreach \i in {#2} {
                    \node (n\i) at (c\i) [dormant] {};
                }
            }
            \begin{scope}[shift={(0,0)}]
                \DrawVertices{3,4,5,6,7}{1,2}
                \draw[dashed] (n1) -- (n2);
                \draw (n3) -- (n4) -- (n5) -- (n6) -- (n7) -- (n3);
            \end{scope}
            \begin{scope}[shift={(5.25,0)}]
                \DrawVertices{1,2,5,6,7}{3,4}
                \draw (n1) -- (n2);
                \draw (n5) -- (n6) -- (n7);
                \draw [dashed] (n7) -- (n3) -- (n4) -- (n5);
            \end{scope}
            \begin{scope}[shift={(0,-3.25)}]
                \DrawVertices{1,2,5,6,7}{3,4}
                \draw (n1) -- (n2);
                \draw (n5) -- (n6) -- (n7);
                \draw [dashed] (n7) -- (n3) -- (n4) -- (n5);
                \draw (n1) -- (n7);
                \draw (n2) -- (n5);
            \end{scope}

            \begin{scope}[shift={(5.25,-3.25)}]
                \DrawVertices{1,2,3,6,7}{4,5}
                \draw [dashed] (n2) -- (n5) -- (n4) -- (n3);
                \draw [dashed] (n5) -- (n6);
                \draw (n2) -- (n1) -- (n7) -- (n3);
                \draw (n6) -- (n7);
            \end{scope}
        \end{tikzpicture}
        \caption{\(C_5\) is not a special subgraph}
        \label{fig:prop:cycles-not-special:1}
    \end{figure}

\end{proof}

We end with an example similar to Lemma \ref{lem:K_3-cup-K_1-is-special}. Observe that deleting a sufficiently large matching from \(\sqcup_{i=1}^{n+1}K_{p_i}\) when each \(p_i\) is odd always results in a graph isomorphic to \(K_3\sqcup n K_1\). We express this construction precisely in Example \ref{exmp:K_3cupnK_1-is-special} below.
\begin{exmp}
\label{exmp:K_3cupnK_1-is-special}
Let \(n\) be a nonnegative integer. 
The sequence \(\{\Gamma_m\}_{m=1}^{\infty}\) is \((K_3 \sqcup nK_1)\)-residual if
for each positive integer \(m\), we have \(\Gamma_m = \bigsqcup_{i=1}^{n+1} K_{p_i}\) where every \(p_i\) is odd and \(\sum_{i=1}^{n+1} p_i = 2(m + n) + 2\).
\end{exmp}
A similar argument to the one used in Lemma \ref{lem:K_3-cup-K_1-is-special} can potentially be used to show the converse.
\begin{conj}
If \(\{\Gamma_m\}_{m=1}^{\infty}\) is \((K_3 \sqcup nK_1)\)-residual, then each \(\Gamma_m\) is of the form described in Example \ref{exmp:K_3cupnK_1-is-special}.
\end{conj}

    \printbibliography

@misc{appiah2024algebraicstructurehyperbolicgraph,
      title={The algebraic structure of hyperbolic graph braid groups}, 
      author={B. Appiah and P. Dani and W. Ge and C. Hudson and S. Jain and M. Lemoine and J. Murphy and J. Murray and A. Pandikkadan and K. Schreve and H. Vo},
      year={2024},
      eprint={2403.08623},
      archivePrefix={arXiv},
      primaryClass={math.GR},
      url={https://arxiv.org/abs/2403.08623}, 
}

@phdthesis{abrams2000configurationspaces,
  author={Abrams, Aaron D.},
  year={2000},
  title={Configuration spaces and braid groups of graphs},
  journal={ProQuest Dissertations and Theses},
  pages={67},
  isbn={978-0-599-85818-3},
  language={English},
  url={https://www.proquest.com/dissertations-theses/configuration-spaces-braid-groups-graphs/docview/304583880/se-2},
  school={University of California, Berkeley}
}

@mastersthesis{ison2005two,
  title={Two Aspects of Topology in Graph Configuration Spaces},
  author={Ison, Molly Elizabeth},
  year={2005},
  school={Virginia Tech},
  url={http://hdl.handle.net/10919/29214}
}

@article{ko2012characteristics,
  title = {Characteristics of Graph Braid Groups},
  author = {Ko, Ki Hyoung and Park, Hyo Won},
  date = {2011-01-13},
  eprint = {1101.2648},
  eprinttype = {arXiv},
  eprintclass = {math},
  doi = {10.48550/arXiv.1101.2648},
  url = {http://arxiv.org/abs/1101.2648},
  pubstate = {prepublished},
}

@article{Maciazek2019,
  author    = {Maciazek, Tomasz and Sawicki, Adam},
  title     = {Non-abelian Quantum Statistics on Graphs},
  journal   = {Communications in Mathematical Physics},
  year      = {2019},
  volume    = {371},
  number    = {3},
  pages     = {921--973},
  doi       = {10.1007/s00220-019-03583-5}
}

@article{PRUE2014136,
title = {Abrams's stable equivalence for graph braid groups},
journal = {Topology and its Applications},
volume = {178},
pages = {136-145},
year = {2014},
issn = {0166-8641},
doi = {https://doi.org/10.1016/j.topol.2014.09.009},
url = {https://www.sciencedirect.com/science/article/pii/S0166864114003836},
author = {Paul Prue and Travis Scrimshaw}
}

@article{GHRIST2007302,
title = {The geometry and topology of reconfiguration},
journal = {Advances in Applied Mathematics},
volume = {38},
number = {3},
pages = {302-323},
year = {2007},
issn = {0196-8858},
doi = {10.1016/j.aam.2005.08.009},
url = {https://www.sciencedirect.com/science/article/pii/S0196885806001175},
author = {Robert Ghrist and V. Peterson},
}

@online{birmanBraidsSurvey2005,
  title = {Braids: {{A Survey}}},
  shorttitle = {Braids},
  author = {Birman, Joan S. and Brendle, Tara E.},
  date = {2005-02-26},
  eprint = {math/0409205},
  eprinttype = {arXiv},
  doi = {10.48550/arXiv.math/0409205},
  url = {http://arxiv.org/abs/math/0409205},
  pubstate = {prepublished},
}

@book{hatcherAlgebraicTopology2001,
  title = {Algebraic {{Topology}}},
  author = {Hatcher, Allen},
  date = {2001},
  publisher = {Cambridge University Press},
  url = {https://pi.math.cornell.edu/~hatcher/AT/ATpage.html},
  isbn = {0-521-79540-0},
}

@article{jainManifoldModelsHyperbolic2026,
  title = {Manifold Models for Hyperbolic Graph Braid Groups on Three Strands},
  author = {Jain, Saumya and Vo, Huong},
  date = {2026-03-08},
  eprint = {2603.07807},
  eprinttype = {arXiv},
  eprintclass = {math},
  doi = {10.48550/arXiv.2603.07807},
  url = {http://arxiv.org/abs/2603.07807},
  pubstate = {prepublished},
  version = {1},
}

@article{ghristSafeCooperativeRobot2002,
  title = {Safe {{Cooperative Robot Dynamics}} on {{Graphs}}},
  author = {Ghrist, Robert W. and Koditschek, Daniel E.},
  date = {2002-01},
  journaltitle = {SIAM Journal on Control and Optimization},
  shortjournal = {SIAM J. Control Optim.},
  volume = {40},
  number = {5},
  pages = {1556--1575},
  issn = {0363-0129, 1095-7138},
  doi = {10.1137/S0363012900368442},
  url = {http://epubs.siam.org/doi/10.1137/S0363012900368442},
  langid = {english},
}

@thesis{fernandesTopologyGraphConfiguration,
  title = {Topology of {{Graph Configuration Spaces}}},
  author = {Fernandes, Praphat},
  date = {2003},
  school = {Virginia Tech},
  langid = {english},
}

@article{plummerNextendableGraphs1980,
  title = {On N-Extendable Graphs},
  author = {Plummer, M. D.},
  date = {1980-01-01},
  journaltitle = {Discrete Mathematics},
  shortjournal = {Discrete Mathematics},
  volume = {31},
  number = {2},
  pages = {201--210},
  issn = {0012-365X},
  doi = {10.1016/0012-365X(80)90037-0},
  url = {https://www.sciencedirect.com/science/article/pii/0012365X80900370},
}

@article{chenRemovableForcedSubgraphs2024,
  title = {Removable and Forced Subgraphs of Graphs},
  author = {Chen, Wuxian and Zhang, Heping},
  date = {2024-07-15},
  journaltitle = {Discrete Applied Mathematics},
  shortjournal = {Discrete Applied Mathematics},
  volume = {351},
  pages = {23--35},
  issn = {0166-218X},
  doi = {10.1016/j.dam.2024.03.004},
  url = {https://www.sciencedirect.com/science/article/pii/S0166218X24001021},
  urldate = {2026-06-10},
}

@article{plummerRecentProgressMatching2008,
  title = {Recent {{Progress}} in {{Matching Extension}}},
  booktitle = {Building {{Bridges}}: {{Between Mathematics}} and {{Computer Science}}},
  author = {Plummer, Michael D.},
  editor = {Grötschel, Martin and Katona, Gyula O. H. and Sági, Gábor},
  date = {2008},
  pages = {427--454},
  publisher = {Springer},
  location = {Berlin, Heidelberg},
  doi = {10.1007/978-3-540-85221-6_14},
  url = {https://doi.org/10.1007/978-3-540-85221-6_14},
}
\end{document}